\documentclass{article}
\usepackage{amsfonts}
\usepackage{amsmath}
\usepackage{amsthm}
\usepackage{harvard}
\usepackage{dsfont}
\usepackage{verbatim}
\usepackage{cite}
\usepackage{xcolor}
\usepackage{bm}

\usepackage[
a4paper,
left=2.5cm,
right=2.5cm,
top=2.5cm,
bottom=3.5cm
]{geometry}
\newtheorem{theorem}{Theorem}

\newtheorem{condition}[theorem]{Condition}

\newtheorem{definition}[theorem]{Definition}

\newtheorem{lemma}[theorem]{Lemma}

\newtheorem{problem}[theorem]{Problem}
\newtheorem{proposition}[theorem]{Proposition}
\newtheorem{remark}[theorem]{Remark}

\numberwithin{equation}{section}
\numberwithin{theorem}{section}

\begin{document}

\title{Grey Random Fields on the Sphere}
\author{L. Beghin, N. Leonenko, D. Marinucci}
\maketitle

\begin{abstract} 

We introduce a new and tractable class of non-Gaussian isotropic random fields on the sphere, which we call, for brevity, \textit{Grey Spherical Random Fields} (GSRFs). The proposed framework provides a flexible extension of the classical Gaussian theory, while preserving strong isotropy and embedding Gaussian random fields as a limiting special case.

The construction is based on random spherical harmonic coefficients obtained as Gaussian mixtures driven by multivariate subordinators evaluated at independent random times (with Mittag-Leffler or more general distribution). This representation allows us to derive explicit characteristic functions, mixed moments and cumulants of arbitrary order. We also investigate asymptotic high-frequency behaviour and identify general conditions under which suitably normalized spherical harmonic coefficients converge to Gaussian scale mixtures, recovering several important models, including stable, Gamma and inverse Gaussian cases, as particular examples.
Proving that GSRFs emerge as scaling limits of multivariate continuous-time random walks, with heavy-tailed waiting times, provides a microscopic foundation for the proposed models.

We further obtain quantitative bounds measuring the deviation from Gaussianity of the spherical harmonic coefficients. The resulting framework substantially broadens the class of analytically tractable isotropic random fields on the sphere, while preserving explicit probabilistic representations and flexible dependence structures.


\end{abstract}

\begin{itemize}
    \item \textit{Keywords and Phrases}: Spherical Random Fields, Isotropy, Grey Random Fields, Non-Gaussianity, Subordinators.
    \item \textit{AMS Classification}: 60G60, 60G51, 33E15.
\end{itemize}

\section{Introduction}

Over the last two decades, isotropic spherical random fields have attracted
a lot of interest and an enormous number of applications in a variety of different areas, including for instance Cosmology, Geophysics, Quantum Mechanics,
Machine Learning and Neural Networks; more recently, spherical random fields which evolve over time have also been considered, in view of important motivations arising in particular from Climate Science. 
The overwhelming
majority of this literature has focused on Gaussian circumstances; this is also due to the mathematical difficulty to introduce principled non-Gaussian models in these settings.

In particular, the main tool for studying finite variance isotropic spherical
random fields is the spectral representation theorem, which allows one to
express a spherical random field $\mathcal{T}(.):\mathbb S^{2} \times \Omega$ as%
\begin{equation}\label{eq:SpecRap}
\mathcal{T}(x,\omega)=\sum_{\ell =0}^{\infty }\sum_{m=-\ell }^{\ell }a_{\ell m}(\omega)Y_{\ell m}(x)%
\text{ , }x\in \mathbb S^{2},
\end{equation}%
where equality holds in both the $L^{2}(\Omega )$ and the $%
L^{2}(\Omega \times \mathbb S^{2})$ sense, the $Y_{\ell m}(.)$ is an orthonormal
system of eigenfunctions of the spherical Laplacian operator (\cite{YAD,Leo1999,MP11}), and
the coefficients $\left\{ a_{\ell m}\right\} $ are zero mean and
uncorrelated, with variance $E[a_{\ell m}^{2}]=C_{\ell }$, the latter
defining the angular power spectrum of the random field. Note that in this paper we shall always take the spherical harmonics and the corresponding random coefficients as real valued, without loss of generality.

The problem of
constructing isotropic non-Gaussian models starting from this representation can
be explained as follows: it can be shown that the coefficients $\left\{
a_{\ell m}\right\}_{\ell,m=-\ell,...,\ell}$ are independent if and only if the field is Gaussian,
hence sampling them independently from any other law will necessarily lead
to anisotropic models. For this reason, basically all forms of
non-Gaussianity considered in the literature have been obtained exploiting so-called ``local"
models, i.e. nonlinear pointwise transforms of the field in the real domain.
For instance, in the Cosmological community the benchmark model for
non-Gaussianity takes the form%
\begin{equation*}
\mathcal T_{NG}(x)=\mathcal T_{G}(x)+f_{NL}\left\{\mathcal T_{G}^{2}(x)-1\right\} \text{ ,}
\end{equation*}%
where $\left\{ \mathcal T_{G}(.)\right\} $ is a zero mean unit variance Gaussian
random field with a given angular power spectrum and $f_{NL}$ is a
nonlinearity parameter (see \cite{BKMR04,Planck18} and the references therein). In the regime where $f_{NL}$ is taken to be ``small",
it is possible to obtain some approximate expression for the spherical
harmonic coefficients of the field $\mathcal T_{NG}(x)$ and their joint cumulants of
order 3 and 4 (called bispectrum and trispectrum). These models have attracted an
enormous amount of interest in the Cosmology, leading to
hundreds of published papers, due to the simplicity and physical relevance of the resulting
non-Gaussianities: from the mathematical point of view, however, it cannot be
denied that these constructions are rather specific and hard to generalize. In particular, little is known about the properties of the corresponding spherical harmonic coefficients: their density or characteristic functions are virtually impossible to compute analytically. There have been other attempts to introduce non-Gaussian models, in particular considering weighted forms of Poisson random waves; see \cite{BDMT,GM2026}; once again, although third and fourth-order moments can be computed explicitly, these models do not allow any analytic derivation on the law of the spherical harmonic coefficients, other than convergence to a limiting Gaussian field in a suitable asymptotic limit.


\subsection{Overview}

Our purpose in this paper is to introduce a general construction which
allows us to generate much more general non-Gaussian models, at the same time allowing for Gaussianity as a special case. In order to do so, we shall need to specify a joint distribution for the triangular array of random coefficients 
$\{a_{\ell,m}\}_{(\ell=0,1,2,...,m=-\ell,...,\ell)}$, ensuring isotropy, i.e., invariance with respect to the action of the rotation group $SO(3)$. We will achieve this goal by introducing the class of 
\emph{Grey Spherical Random Fields}. For these models, we will provide a fully explicit form of the joint characteristic function for the spherical harmonic coefficients across different values of $\{\ell,m\}$; we will show that they cover Gaussianity as a special case. The definition ensures that strong isotropy is maintained and allows for an explicit computation of higher order moments and joint cumulants. The same construction is considered below also for random fields which do not have finite variance; similar ideas could be easily extended to cover sphere cross-time random fields, although we leave this issue for future research. 

In the sequel, we introduce this new class of random fields moving step by step towards greater and greater generality. In particular, our construction will depend on two ingredients: a general form of the (grey) characteristic functions, in terms of Mittag-Leffler functions and depending on a positive parameter $\beta$ and embedding the Gaussian case as a special one for $\beta=1$, and a general construction of L\'evy probability measures, which allow for an interpretation in terms of subordinators (see \cite{APP}, for general reference to the theory of subordinators, and \cite{GRO} for Mittag-Leffler analysis and grey Brownian motion, among many others). To model these subordinators we will allow broader circumstances, starting from i.i.d. random variables in the simplest construction and allowing instead for very broad, mutually dependent vector processes with random time in the most general case. Furthermore, we will establish rigorous quantitative upper bounds on the distance between the law of the random spherical harmonic coefficients in the Gray case and the multivariate Gaussian law; interestingly, these bounds depend on two additive components, one of which allows an interpretation in terms of the amount of ``Grayness" in the characteristic function, the other in terms of the amount of dependence in the coefficients (again, recall that random spherical harmonics are independent if and only if the field is Gaussian).

A further contribution of the paper is a probabilistic microscopic interpretation of the proposed models. We prove that GSRFs arise naturally as scaling limits of multivariate continuous-time random walks with heavy-tailed waiting times, providing a rigorous link between microscopic random spectral energy-transfer mechanisms and macroscopic non-Gaussian random fields on the sphere.

The proposed framework considerably enlarges the class of analytically tractable isotropic random fields available in the literature, combining explicit probabilistic representations, flexible dependence structures across harmonic scales, and a natural interpretation in terms of anomalous stochastic dynamics.

Grey Brownian motion and its generalizations provide one of the best-known examples of non-Gaussian stochastic processes whose finite-dimensional distributions are governed by Mittag-Leffler characteristic functions (starting from the paper \cite{SCH} by Schneider). These models have attracted considerable attention in the study of anomalous diffusion and fractional stochastic dynamics, thanks to their explicit probabilistic representations through random times or Gaussian scale mixtures driven by non-negative random variables; see, for example, \cite{GRO}, \cite{MUR}. Alternative forms of the Grey Brownian motions, defined by means of the incomplete Gamma function or completely monotone functions, more general than the Mittag-Leffler one, can be found in \cite{BCG} and \cite{BCDS}.

The construction proposed in this paper can be viewed as a natural extension of this framework from stochastic processes to isotropic random fields on the sphere, preserving the analytical tractability of Grey distributions while allowing for a rich dependence structure across harmonic frequencies.

Our definition is thus conceptually different from the fractional spherical random fields introduced in \cite{DOV}: while in their framework fractional operators act on an underlying Gaussian spherical field and modify its spectral dynamics through the eigenvalues of the spherical Laplacian, our GSRFs are intrinsically non-Gaussian, with non-Gaussianity and cross-multipole dependence arising directly from Gaussian variance mixtures driven by a multivariate subordinator, whose joint Laplace exponents encode the dependence structure among different harmonic scales.

\subsection{Outline of the paper}

In Section \ref{sec:Background} we collect some background notation and results which are used throughout this work, whereas Section \ref{sec:GreyGen} introduces the general class of GSRFs and establishes thier probabilistic representation in terms of multivariate subordinators evaluated at an independent Mittag-Leffler random time. Section \ref{sec:GreyInd}
is devoted to our construction, in the case of independent subordinators, for which we derive explicit expressions of moments and cumulants, together with high-frequency limit results for the spherical harmonic coefficients. The more general dependent circumstances are dealt with in Section \ref{secdepsub}, where we provide interpretation of the random fields as the scaling limit of multivariate continuous-time random walks, with heavy-tailed waiting times. Explicit examples within the class of tempered stable subordinators are also discussed. Section \ref{sec:bounds} includes our main results on the quantitative bounds in the distance to Gaussianity, as an explicitly function of some Greyness/dependence parameters to be introduced before; Section \ref{sec:GGSRF} considers again the simpler case of independent and identically distributed subordinators, but under much more general, non-Gaussian characteristic functions. Section \ref{sec:space-time} introduces a space-time extension of GSRFs and derives the time-fractional evolution equation governing the joint law of their spherical harmonic coefficients.
A number of technical results and applications of the previous models to relevant special cases are collected in the Appendices.

\section{Background and Notation}\label{sec:Background}
\subsection{Aims of the Paper}
In order for a spherical random field to be isotropic, i.e. invariant in law to rotations, its spherical harmonic coefficients $a_{\ell,m}, \ell=0,1,2,...,m=-\ell,...,\ell$ must satisfy some rather tight constraints. In particular, it is well-known that a random field is (strongly) isotropic if and only if, for each value of $\ell$, the following invariance conditions are met:
\begin{equation}\label{InvCon}
\boldsymbol{a}_{\ell,\cdot}\overset{d}{=}D_{\ell}(g)\boldsymbol{a}_{\ell,\cdot}\ , \text{ for each } g \in SO(3) \ ,
\end{equation}
where we wrote $\boldsymbol{a}_{\ell,\cdot}$ for the column vector $(a_{\ell,-\ell},...,a_{\ell,-\ell})^T$, whereas the $(2\ell+1) \times (2\ell+1)$ matrices $D_{\ell}(g)$ represent the so-called Wigner's representation of the compact group $SO(3)$, see \cite{MP11} for more discussion and details (the statement holds for complex-valued spherical harmonics coefficients, but it can be easily adapted to the real-valued expansion that we use in this paper). Our aim is to build a non-Gaussian model which satisfies the following three conditions:

\begin{itemize}
\item It must satisfy the isotropy condition

\item It must embed Gaussian models as a special case

\item It must allow for an analytic representation of the law of its spherical harmonic coefficients, either in terms of their density or their characteristic functions.
\end{itemize}

In order to achieve these goals, our plan is to introduce a class of models  with fully explicit characteristic functions for the random spherical harmonic coefficients; these characteristic functions will depend only on the norm of the random spherical harmonic coefficients for fixed ``multipole" $\ell$, and they will hence be invariant to rotations. Nevertheless, they will allow for a rich, non-trivial structure of dependence across different multipoles. To introduce this class of functions, we must now review some facts in the recent literature on Grey Processes (see \cite{SCH}, \cite{GRO}, \cite{MUR}, among many others), which will provide the building blocks for our construction. 

\subsection{Background on Special Functions and Subordinators}

In this subsection, we review some standard facts about special functions, which are important tools in the definition of Grey Processes; we refer to the books \cite{GOR} and \cite{KIL}, for more details and discussion. 

We start by recalling the  Mittag-Leffler function, which is defined as \begin{equation}
    E_\beta(s):=\sum_{j=0}^\infty \frac{s^{j}}{\Gamma(\beta j+1)},  \qquad s,\beta \in \mathbb{C}. \label{ML}
\end{equation}

The following integral representation over  the Hankel path $Ha$ (i.e. a loop that starts and ends at $-\infty$ and encircles the circular disk $|\omega|\leq |s|^{1/\beta}$: $-\pi \leq arg \omega \leq \pi$) can be easily obtained, by resorting to the contour integral representation of the reciprocal Gamma function:
\begin{eqnarray}
  E_{\beta}(s)&=& \frac{1}{2\pi i} \sum_{j=0}^\infty s^{j} \int_{Ha}e^\omega \omega^{-1-\beta j}d\omega = [\text{for } |\omega|>|s|^{1/\beta} ]\label{ml4}\\ 
  &=& \frac{1}{2\pi i}\int_{Ha}e^\omega \frac{\omega^{-1}}{1-s/\omega^{\beta}}d\omega =\frac{1}{2\pi i}\int_{Ha}\frac{e^\omega \omega^{\beta-1}}{\omega^\beta -s}d\omega, \qquad s \in \mathbb{C}. \notag
\end{eqnarray}

It is well-known that the function $ E_{\beta}(-s)$, for $s \in \mathbb{R}_+$, $\beta \in (0,1)$, is completely monotone (i.e. it is $C^\infty$ and such that $(-1)^n E_\beta^{(n)}(-s)\geq 0, $ for any $n \geq 1$ and $s\geq 0$) and it coincides with the Laplace transform of the Mainardi function (see e.g. \cite{GOR})
\begin{equation}
    M_\beta(x):=\sum_{j=0}^\infty\frac{(-x)^j}{j! \Gamma(-\beta j+1-\beta)}=\frac{1}{2\pi i}\int_{Ha}\frac{e^{\sigma-x\sigma^\beta}}{\sigma^{1-\beta}}d\sigma, \qquad x \geq 0 \label{mai}
\end{equation}
as can be also checked directly from \eqref{ml4} and \eqref{mai}. The function  \eqref{mai} is the density of an absolutely continuous random variable $T_\beta$ with finite moments given by 
\begin{equation}\label{mom}
\mathbb{E}T_\beta^\delta=\frac{\Gamma(\delta+1)}{\Gamma(\beta \delta +1)},  \qquad \delta\geq -1,
\end{equation}
(see \cite{GOR}).

Finally, we recall that the Mittag-Leffler function (with parameter $\beta$) coincides with the eigenfunction of the Caputo fractional derivative of order $\beta$, i.e.
\begin{equation}
{}^C D_t^\beta
E_\beta(-\lambda t^\beta)
=
-\lambda E_\beta(-\lambda t^\beta),\qquad t \geq 0,\lambda \in \mathbb R,\label{Caputo}
\end{equation}
where \begin{equation}
{}^C D_t^\beta
f(t):
=
\begin{cases}
    \frac{1}{\Gamma(1-\beta)}\int_0^tf'(s)(t-s)^{-\beta} ds, \qquad \beta \in (0,1), \notag \\
    f'(t), \qquad \beta =1.
\end{cases}\label{Caputo2}
\end{equation}
Let us now also recall that a subordinator $\mathcal{H}:=\left\{ \mathcal{H}(t) \right\}_{t \geq 0}$ is an a.s. non-decreasing L\'evy process, such that $\mathcal H(0)=0$ a.s., whose Laplace transform reads
\[\mathbb{E}e^{-\kappa\mathcal{H}(t)  }=e^{-t\varphi(\kappa)}, \qquad \kappa >0, t \geq 0,\]
where $\varphi(\cdot)$ is a Bernstein function (i.e. is such that $(-1)^{n-1} \varphi^{(n)}(\kappa)\geq 0, $ for any $n \geq 1$ and $\kappa > 0$), which is called Laplace exponent of $\mathcal{H}$ (see e.g. \cite{APP}). In the multivariate case, we define the $d$-dimensional subordinator $\boldsymbol{\mathcal H}:=\left\{(\mathcal{H}_{1}(t),...,\mathcal{H}_{d}(t))\right\}_{t \geq 0}$ as the L\'evy process with Laplace exponent $\varphi(\kappa_1,...,\kappa_k)$, i.e. such that
\begin{equation}\mathbb{E}e^{-\kappa_1 \mathcal{H}_1(t)...-\kappa_d \mathcal{H}_d(t)  }=e^{-t\varphi(\kappa_1,...,\kappa_d)}, \qquad \kappa_i >0, i=1,...,d, \text{ and } t \geq 0, \label{lapsub}\end{equation}
where $\varphi:\mathbb{R}_+^d \rightarrow \mathbb{R}_+$ is a multidimensional Bernstein function, i.e. there exist constants $a\geq 0$, $\boldsymbol{b} \in \mathbb{R}_+^d$ and a measure $\mu(\cdot)$ on $\mathbb{R}_+^d$ verifying $\int_{\mathbb{R}_+^d}\Vert \boldsymbol{s} \Vert_d/(1+\Vert \boldsymbol{s} \Vert_d) \mu(\boldsymbol{ds}) < \infty$, such that
\[
\varphi(\boldsymbol{x})=\boldsymbol{a}+\langle \boldsymbol{b} ,\boldsymbol{x} \rangle+\int_{\mathbb{R}^d_+}(1-e^{-\langle \boldsymbol{x},\boldsymbol{s} \rangle})\mu(\boldsymbol{ds}), \qquad \boldsymbol{x} \in \mathbb{R}_+^d,
\]
where $\Vert \boldsymbol
x\Vert_d:=\sqrt{x_1^2+...+x_d^2} $ (see e.g. \cite{BAR}). We also recall that a multivariate function $f:\mathbb{R}_+^d \rightarrow \mathbb{R}$ is completely monotone if it is $C^\infty$ and such that $(-1)^{n_1+...+n_d}\frac{\partial^{n_1+...+n_d}f }{\partial x^{n_1}\cdots\partial x^{n_d}} \geq 0$, for any combination of $n_1,...,n_d \geq 0$  (see \cite{BOC}, p.87). Then the following (unique representation) holds:
\[
f(\boldsymbol{x})=\int_{\mathbb{R}_+^d}e^{-\langle \boldsymbol{x},\boldsymbol{s} \rangle}\mu(\boldsymbol{ds}), \qquad \boldsymbol{x} \in \mathbb{R}_+^d.
\]

\section{Grey Random Fields on $\mathbb{S}^{2}$: definitions and generalities} \label{sec:GreyGen}

In order to implement our construction of Grey spherical random fields (hereafter, GSRFs), we need first to introduce the following multivariate function, which we prove to be completely monotone in our first result below.

\begin{definition}\label{defsub}
    Let $\varphi: \mathbb{R}^d_+ \rightarrow \mathbb{R}_+$ be the Laplace exponent of the multivariate subordinator  $\boldsymbol{\mathcal{H}}$,
   i.e. 
\[
\varphi(\boldsymbol{\kappa})
=
\langle
\boldsymbol b,
\boldsymbol{\kappa}
\rangle
+
\int_{\mathbb R_+^d\setminus\{0\}}
\left(
1-e^{-\langle\boldsymbol{\kappa},\boldsymbol{x}\rangle}
\right)
\nu(d\boldsymbol{x}),
\]
for $\boldsymbol{b}\in \mathbb R^d$ and $\nu(\cdot)$ a L\'evy measure on $\mathbb R^d_+\setminus\{0\}$, with \(
\int_{\mathbb R_+^d}
(1\wedge \|x\|)
\,\nu(dx)
<\infty.
\)
    Then, for any $\beta \in (0,1]$, we define

    \begin{equation}
\Psi_\beta(\kappa_1,..,\kappa_d):=E_\beta(-\varphi(\kappa_1,...,\kappa_d))=\sum_{j=0}^\infty \frac{\left(-\varphi(\kappa_1,...,\kappa_d)\right)^j}{\Gamma(\beta j+1)}, \label{psi}
    \end{equation}
    for $\beta \in (0,1]$ and $ \kappa_j\geq 0$, $j=1,..,d$.
\end{definition}

\begin{lemma}\label{theocm}
The function $\Psi_\beta:\mathbb{R}_+^d \rightarrow \mathbb{R}_+$ defined in \eqref{psi} is completely monotone and coincides with the Laplace transform of $(\mathcal{H}_1(T_\beta),...,\mathcal{H}_d(T_\beta))$, where $T_\beta$ is a r.v., independent from $\left\{(\mathcal{H}_{1}(t),...,\mathcal{H}_{d}(t))\right\}_{t \geq 0}$ and with density function  $M_\beta(\cdot)$ given in \eqref{mai}.
\end{lemma}
\begin{proof}
By a conditioning argument, we can write
\begin{eqnarray}
  \mathbb{E}e^{-\kappa_1  \mathcal{H}_1(T_\beta)...-\kappa_d \mathcal{H}_d(T_\beta)  }&=&  \int_{0}^\infty\mathbb{E}e^{-\kappa_1 \mathcal{H}_1(t)...-\kappa_d  \mathcal{H}_d(t)  }M_\beta(t)dt\notag\\
 &=&\int_{0}^\infty e^{-t\varphi(\kappa_1,...,\kappa_d)}M_\beta(t)dt=E_\beta(-\varphi(\kappa_1,...,\kappa_d)), \notag
\end{eqnarray}
for $\kappa_i >0, i=1,...,d,$ by recalling that $E_\beta(-s)=\int_0^\infty e^{-sx}M_\beta(x)dx$. As a consequence, the following representation holds for the function defined in \eqref{psi}:
\begin{equation}
\Psi_\beta(\kappa_1,..,\kappa_d)=\int_{\mathbb{R}_+^d}e^{-\langle \boldsymbol{x},\boldsymbol{s} \rangle}\mu(\boldsymbol{ds}), \qquad \boldsymbol{x} \in \mathbb{R}_+^d,
\end{equation}
where $\mu(\cdot)$ is the joint distribution of $(\mathcal{H}_1(T_\beta),...,\mathcal{H}_d(T_\beta))$.
\end{proof}
\begin{remark}
In the special case $\beta=1$, the function $\Psi_\beta(\cdot,...,\cdot)$ coincides with the Laplace transform of the density of the r.v. $(\mathcal{H}_{1}(1),...,\mathcal{H}_{d}(1))$, i.e. $\Psi_1(\kappa_1,...,\kappa_d)=e^{-\varphi(\kappa_1,..,\kappa_d)}$, since, in this case, $M_\beta(x)=\delta(x-1)$, where $\delta(\cdot)$ is the Dirac delta function with pole in the origin. 
\end{remark}
We will distinguish three cases, with increasing degrees of generality, according to the assumptions on the multivariate distribution and dependence structure of the subordinators $(\mathcal{H}_1(T_\beta),...,\mathcal{H}_d(T_\beta))$ in Def. \ref{defsub}. 

More precisely, we will say that 

\begin{definition}\label{defGRF}
    
A spherical random field with power spectrum $\left\{ C_{\ell
}\right\} $ is Grey if for any choice of multipoles $\ell_1,...,\ell_d$ we have that%
\begin{equation}\label{def:GeneralGF}
\Phi _{\boldsymbol{a}_{\ell _{1\cdot}},...,\boldsymbol{a}_{\ell
_{d\cdot}}}(\boldsymbol{\xi}_{\ell_1\cdot},\boldsymbol{\xi}_{\ell_2\cdot},...,\boldsymbol{\xi}_{\ell_d\cdot})=E_\beta\left(-\varphi_{\ell_1,...\ell_d}(C_{\ell_1}
\Vert\boldsymbol \xi_{\ell_1\cdot}\Vert^2/2,...,C_{\ell_d}
\Vert\boldsymbol \xi_{\ell_d\cdot}\Vert^2/2)\right),
\end{equation}%
for $\beta \in (0,1)$, $\xi_i \in\mathbb{R}$, $i=1,...,d$, and $\{\varphi_{\ell_1,...\ell_d}(.,...,.)\}_{\ell_1,...,\ell_d \in \mathbb{N}^d}$ a family of multivariate Laplace exponent functions.
\end{definition}

We will discuss below what are the general properties that must be fulfilled for the family of functions $\{\varphi_{\ell_1,...\ell_d}(.,...,.)\}_{\ell_1,...,\ell_d \in \mathbb{N}^d}$ to represent an admissible ensemble of Laplace exponents. 

Let us list the three cases that we will adress below.

\textbf{Case 1:} We say that a Grey spherical random field has random global variance if it holds that
\[
\varphi_{\ell_1,...\ell_d}(\kappa_1,...,\kappa_d)=\sum_{i=1}^d \kappa_i,
\]
for all choices of ${\ell_1,...,\ell_d \in \mathbb{N}^d}$.

\textbf{Case 2:} We say that a Grey spherical random field is generated by independent subordinators if it holds that
\[
\varphi_{\ell_1,...\ell_d}(\kappa_1,...,\kappa_d)=\sum_{i=1}^d\varphi_{\ell_i}(\kappa_i),
\]
for all choices of ${\ell_1,...,\ell_d \in \mathbb{N}^d}$. The rationale behind this terminology is going to be explained in the next Section below.

\textbf{Case 3:} We say that a Grey spherical random field is generated by dependent subordinators if for some ${\ell_1,...,\ell_d \in \mathbb{N}^d}$, the Laplace exponent can not be written as the sum of univariate functions.

Clearly a Grey random field with random global variance is a special case of the fields generated by independent subordinators; i.e. Case 1 is mathematically included in Case 2. However, for reasons to be explained in the section, however, it seems to deserve a special mention.

\section{GSRFs with independent subordinators} \label{sec:GreyInd}

As anticipated in the previous Section, by a simple computation it is immediate to see that, if the subordinator vector $\left\{(\mathcal{H}_{1}(t),...,\mathcal{H}_{d}(t))\right\}_{t \geq 0}$ has independent (though not-necessarily identically distributed) components, we have that
\begin{equation}\label{psi2}
    \Psi_\beta(\kappa_1,..,\kappa_d)=E_\beta(-\varphi_1(\kappa_1)-...-\varphi_d(\kappa_d))=\sum_{j=0}^\infty \frac{\left(-\varphi_1(\kappa_1)-...-\varphi_d(\kappa_d)\right)^j}{\Gamma(\beta j+1)}, 
\end{equation}
where $\varphi_i(\cdot)$, $i=1,...,d$, are the Bernstein functions of $\left\{\mathcal{H}_i(t)\right\}_{t\geq 0}$.

We now need a further bit of notation; we shall write $\boldsymbol{\xi}_{\ell,\cdot}:=(\xi_{\ell,-\ell},...,\xi_{\ell,\ell})$ for the $2\ell+1$ vectors of arguments of the characteristic functions that we will deal with. In particular, from the discussion in the previous Section we have the following characterization of Grey random fields with independent subordinators, in terms of their spherical harmonic coefficients:

\begin{definition}\label{defGRFphi}
    
A spherical random field with power spectrum $\left\{ C_{\ell
}\right\} $ is Grey if for any choice of multipoles $\ell_1,...,\ell_d$ the random spherical harmonic coefficients have the following characteristic function:
\begin{equation}\label{defGRFeq}
\Phi _{\boldsymbol{a}_{\ell _{1\cdot}},...,\boldsymbol{a}_{\ell
_{d\cdot}}}(\boldsymbol{\xi}_{\ell_1\cdot},\boldsymbol{\xi}_{\ell_2\cdot},...,\boldsymbol{\xi}_{\ell_d\cdot})=E_\beta\left(-\sum_{i=1}^d\varphi_i(C_{\ell_i}
\Vert\xi_{\ell_i\cdot}\Vert^2/2)\right),
\end{equation}%
for $\beta \in (0,1)$, $\xi_i \in\mathbb{R}$, $i=1,...,d$.
\end{definition}

We note first that one of our goals is immediately achieved, i.e., we have Gaussian random fields as a special case, see the Remark below.

\begin{remark}(The Gaussian special case) When $\beta=1$ and the Laplace exponents of the subordinator components are $\varphi_i(\kappa)=\kappa$, $\kappa \geq 0$ (i.e. $\mathcal{H}_i(t)=t$ a.s., for any $t \geq 0$), for $i=1,...,d$, we obtain the characteristic function of a Gaussian field:

\begin{equation}\label{defGRFeq2}
\Phi _{\boldsymbol{a}_{\ell _{1\cdot}},...,\boldsymbol{a}_{\ell
_{d\cdot}}}(\boldsymbol{\xi}_{\ell_1\cdot},\boldsymbol{\xi}_{\ell_2\cdot},...,\boldsymbol{\xi}_{\ell_d\cdot})=\exp\left(-\frac{1}{2}\sum_{i=1}^dC_{\ell_i}
\Vert\xi_{\ell_i\cdot}\Vert^2)\right).
\end{equation}
This is actually the only case where the spherical harmonic coefficients are independent (consistently with the results in \cite{BM,BMV}).

\end{remark}

\subsection{GSRFs with random global variance}\label{random global variance}

As the first non-Gaussian special case we will consider the situation where Definition \ref{defGRF} holds and the Laplace exponents of the subordinator components are $\varphi_i(\kappa)=\kappa$, $\kappa \geq 0$ (i.e. $\mathcal{H}_i(t)=t$ a.s., for any $t \geq 0$), for $i=1,...,d$, and thus %

\begin{equation}
\Phi _{\boldsymbol{a}_{\ell _{1\cdot}},...,\boldsymbol{a}_{\ell
_{d\cdot}}}(\boldsymbol{\xi}_{\ell_1\cdot},\boldsymbol{\xi}_{\ell_2\cdot},...,\boldsymbol{\xi}_{\ell_d\cdot})=E_\beta\left(-\frac{1}{2}\sum_{i=1}^dC_{\ell_i}
\Vert\xi_{\ell_i\cdot}\Vert^2\right), \qquad 0<\beta <1\text{.}\label{ml3
}
\end{equation}

As mentioned earlier, in this framework we say that we have a GSRF with random global variance. Indeed, here the spherical harmonic coefficients are still identically distributed after normalizing by their
variance; in fact, although they are not independent (because $\beta$ is strictly less than one) they are still exchangeable, in the classical De Finetti sense. In particular, the characteristic function of a single spherical
harmonic coefficient is given by%
\begin{equation*}
\Phi _{a_{\ell m}}(\xi)=E_{\beta }\left( -\frac{1}{2}C_{\ell }\xi^{2}\right) 
\text{ .}
\end{equation*}%
The characteristic function for the field itself is 
\begin{eqnarray*}
\Phi _{\mathcal{T}(x)}(\xi) &=&\mathbb{E}[\exp (i\xi\sum_{\ell m}a_{\ell m}Y_{\ell m}(x))]
\\
&=&E_{\beta }\left( -\frac{1}{2}\xi^{2}\sum_{\ell m}C_{\ell }Y_{\ell
m}^{2}(x)\right) =E_{\beta }\left( -\frac{1}{2}\xi^{2}\sum_{\ell }\frac{2\ell +1}{4\pi }%
C_{\ell }\right) \text{ .}
\end{eqnarray*}%
The joint characteristic function is given by%
\begin{eqnarray*}
\Phi _{\mathcal{T}(x_{1}),...,\mathcal{T}(x_{d})}(\xi_{1},\xi_{2},...,\xi_{d}) &=&\mathbb{E}[\exp
(i\sum_{i=1}^{d}\sum_{\ell m}\xi_{i}a_{\ell m}Y_{\ell m}(x_{i}))] \\
&=&E_{\beta }\left( -\frac{1}{2}\sum_{i=1}^{d}\sum_{j=1}^{d}\xi_{i}\xi_{j}\sum_{%
\ell m}C_{\ell }Y_{\ell m}(x_{i})Y_{\ell m}(x_{j})\right) \\
&=&E_{\beta }\left( -\frac{1}{2}\sum_{i=1}^{d}\sum_{j=1}^{d}\xi_{i}\xi_{j}\sum_{%
\ell }\frac{2\ell +1}{4\pi }C_{\ell }P_{\ell }(\left\langle
x_{i},x_{j}\right\rangle )\right) \text{ .}
\end{eqnarray*}%
In this special case, we can hence characterize these fields as follows:

\begin{itemize}

    \item after normalization, all the spherical harmonic coefficients have the same law

\item this law is the same as the law of the field itself, after normalization,
at any given location

\item the spherical harmonic coefficients are clearly exchangeable at the same
multipole $\ell$, hence have a uniform law on the sphere $\mathbb{S}^{2\ell+1}$; indeed, using results from \cite{BMT25} it is possible to show that all fields with exchangeable coefficients
have this form, for a given completely monotone characteristic function

\item there is no independence across different multipoles, unless the
characteristic function is exponential in the norm (i.e., the Gaussian case).
\end{itemize}



An important question is on the computation of higher-order moments
and cumulants, that is, polyspectra. We note first that all odd-order polyspectra are easily seen to be zero (by symmetry considerations). For the even-order polyspectra, we start from the fourth-order one;  we have that the reduced trispectrum (i.e. the cumulant of the spherical harmonic coefficients) is
going to be given by%
\begin{eqnarray*}
\mathbb{E[}a_{\ell 0}^{4}]-3(\mathbb{E[}a_{\ell 0}^{2}])^{2} &=&\left. \frac{%
d^{4}}{dt^{4}}\sum_{n=0}^{\infty }\frac{(-\frac{C_{\ell }}{2}t^{2})^{n}}{%
\Gamma (\beta n+1)}\right\vert _{t=0}-3(\mathbb{E[}a_{\ell 0}^{2}])^{2} \\
&=&\frac{C_{\ell }^{2}}{4}\times \frac{24}{2\beta \Gamma (2\beta )}-3(%
\mathbb{E[}a_{\ell 0}^{2}])^{2}.
\end{eqnarray*}%
It is interesting to note that, for $\beta \rightarrow 1$, the fourth-cumulant
tends to zero, as expected because in this case the characteristic function
tends to Gaussian. On the other hand, for $\beta \rightarrow 0$, we have that%
\begin{equation*}
\lim_{\beta \rightarrow 0}2\beta \Gamma (2\beta )=1
\end{equation*}%
and the fourth cumulant tends to $3C_{\ell }^{2}.$ As far as the joint
moments are concerned, we have that, for $m_{1}\neq m_{2}$,%
\begin{eqnarray*}
\mathbb{E[}a_{\ell _{1}m_{1}}^{2}a_{\ell _{2}m_{2}}^{2}] &=&\left. \frac{%
\partial ^{4}}{\partial \xi_{1}^{2}\partial \xi_{2}^{2}}\Phi _{a_{\ell
_{1}m_{1}},a_{\ell _{2}m_{2}}}(\xi_{1},\xi_{2})\right\vert _{\xi_{1}=\xi_{2}=0} \\
&=&\left. \frac{\partial ^{4}}{\partial \xi_{1}^{2}\partial \xi_{2}^{2}}E_{\beta
}\left( -\frac{1}{2}C_{\ell _{1}}\xi_{1}^{2}-\frac{1}{2}C_{\ell
_{2}}\xi_{2}^{2}\right) \right\vert _{\xi_{1}=\xi_{2}=0} \\
&=&\left. \frac{\partial ^{4}}{\partial \xi_{1}^{2}\partial \xi_{2}^{2}}%
\sum_{n=0}^{\infty }\frac{\left( -\frac{1}{2}C_{\ell _{1}}\xi_{1}^{2}-\frac{1}{%
2}C_{\ell _{2}}\xi_{2}^{2}\right) ^{n}}{\Gamma (\beta n+1)}\right\vert
_{\xi_{1}=\xi_{2}=0} =\frac{2C_{\ell _{1}}C_{\ell _{2}}}{\Gamma (2\beta +1)}\text{ .}
\end{eqnarray*}%
Up to a deterministic scaling factor, the cross-moment is exactly the same
regardless of the fact that we focus on the same or different multipoles. This is a consequence of the fact that, with this definition, all the spherical harmonics coefficients
are scaled by the same random variable $\mathcal{H}(T_\beta)$. In the next subsections, we move to more general and richer structures, which allow for more complex dependence among different multipoles.

It should be noted that the construction of GSRFs with random global variance is by no means limited to the case of a Mittag-Leffler characteristic functions; it is indeed possible to introduce a number of other cases in explicit form. Not to interrupt the flow of our argument, we delay the discussion of these other possibilities in detail to Section \ref{sec:GGSRF}.

\subsection{GSRFs with random angular power spectrum}\label{sectangul}

In the non-degenerate case, the GSRFs defined in Def. \ref{defGRFphi} possess random scaling factors that can be different, and in general not identically distributed, for each multipole $\ell$. 

To clarify this point, let $\boldsymbol{Z}_\ell:=\left\{ Z_{\ell m}\right\}_{m=-\ell}^\ell$, $\ell=1,...,L$, be independent random vectors with normal distribution $\boldsymbol{Z}_\ell \sim N(\boldsymbol{0},C_\ell\boldsymbol{I}_{2\ell+1})$. By a conditioning argument, it is simple to prove that
\begin{eqnarray*}\Phi _{\boldsymbol{a}_{\ell _{1\cdot}},...,\boldsymbol{a}_{\ell
_{d\cdot}}}(\boldsymbol{\xi}_{\ell_1\cdot},\boldsymbol{\xi}_{\ell_2\cdot},...,\boldsymbol{\xi}_{\ell_d\cdot})&=&E_\beta\left(-\sum_{i=1}^d\varphi_{\ell_i}(C_{\ell_i}\Vert \boldsymbol{\xi}_{\ell_i}\Vert^2/2)\right)=\mathbb{E}e^{i\sum_{i=1}^d\langle\sqrt{\mathcal{H}_{\ell_i}(T_\beta )}\boldsymbol{Z}_{\ell_i},\boldsymbol{\xi}_{\ell_i}\rangle}\notag\\
&=&\mathbb{E}e^{-\frac{1}{2}\sum_{i=1}^d C_{\ell_i}\thinspace \mathcal{H}_{ \ell_i}(T_\beta)\Vert \boldsymbol{\xi}_{\ell_i}\Vert^2},
\end{eqnarray*}%
for $\beta \in (0,1)$, $\boldsymbol{\xi}_{\ell_i} \in\mathbb{R}^{2\ell_i+1}$, $i=1,...,d$.

More explicitly, we can now write the random spherical harmonic coefficients as  \begin{equation}a_{\ell m}:=\sqrt{\mathcal{H}_{\ell}(T_\beta)}Z_{\ell m}, \qquad\ell=0,1,2,...,m=-\ell,...,\ell . \label{defa}\end{equation} In other words, we can view the spherical harmonic coefficients as different mixtures of Gaussian variables, the mixturing variables being dependent and not identically distributed along the various multipoles. This will clearly yield a much more flexible dependence structure, which we investigate below.

In particular, if the set $(\ell_1,...\ell_d)=(1,...,L)$, we have that
\begin{equation*}E_\beta\left(-\sum_{\ell=1}^L\varphi_\ell(C_{\ell}\Vert \boldsymbol{\xi}_\ell\Vert^2/2)\right)=\mathbb{E}e^{i\sum_{\ell=1}^L\langle\sqrt{\mathcal{H}_\ell(T_\beta )}\boldsymbol{Z}_\ell,\boldsymbol{\xi}_\ell\rangle}=\mathbb{E}e^{-\frac{1}{2}\sum_{\ell=1}^LC_\ell\thinspace \mathcal{H}_{ \ell}(T_\beta)\Vert \boldsymbol{\xi}_\ell\Vert^2}.
\end{equation*}
Note that we assume throughout that
\begin{equation}
\sum_{\ell\ge0}(2\ell+1)C_\ell
\mathbb E\left[\mathcal H_\ell(T_\beta)\right]<\infty .\label{sqrint}
\end{equation}
Under this condition, the series
\[
\mathcal{T}(x)=\sum_{\ell\ge0}\sum_{m=-\ell}^{\ell}
\sqrt{\mathcal H_\ell(T_\beta)}\,Z_{\ell m}Y_{\ell m}(x),\qquad x\in\mathbb S^2,
\]
converges in \(L^2(\Omega\times\mathbb S^2)\), and hence defines a
square-integrable Grey spherical random field. Note that, in the stable case, the assumption \eqref{sqrint} is not satisfied due to the lack of finite moments of the subordinator, and then these ccircumstances are treated separately below.

We will now explore the joint moments $\mathbb{E}(a^{r_1}_{\ell_1 m_1} \cdots a^{r_d}_{\ell_d m_d})$, $r_1,...,r_d \in \mathbb{N}$.

\begin{theorem} \label{mixmom}
The mixed moments of the random spherical harmonic coefficients $\mathbb{E}(a^{r_1}_{\ell_1 m_1}\cdots a^{r_2}_{\ell_d m_d})$ are null, for any $m_i \neq m_j$, if for at least one $i=1,...,d$ it is $r_i=2n_i-1$, $n_i \in \mathbb{N}$, and if $\varphi_i^{(j)}:=\frac{d^j}{d\kappa^j}\varphi_i(\kappa)|_{\kappa=0}< \infty$, $j=1,...,n_i$. On the other hand, 
\begin{eqnarray}
&&\mathbb{E[}a_{\ell _{1}m_{1}}^{2n_1}\cdots a_{\ell _{d}m_{d}}^{2n_d}] \label{moments} \\
&=&C^{n_1}_{\ell_1}\cdots C^{n_d}_{\ell_d}\sum_{\boldsymbol{k}^{(1)}}\cdots \sum_{\boldsymbol{k}^{(d)}}A_{n_1,...,n_d}(\boldsymbol{k}^{(1)},...,\boldsymbol{k}^{(d)})
\prod_{j_1=1}^{n_1}\left(- \varphi_{\ell_1}^{(j_1)}(0)\right)^{k^{(1)}_{j_1}}\cdots\prod_{j_d=1}^{n_d}\left(- \varphi_{\ell_d}^{(j_d)}(0)\right)^{k^{(d)}_{j_d}}, \notag 
\end{eqnarray}
where \begin{eqnarray}
A_{n_1,...,n_d}(\boldsymbol{k}^{(1)},...,\boldsymbol{k}^{(d)}):=\frac{(\sum_{j_1=1}^{n_1} k^{(1)}_{j_1}+...+\sum_{j_d=1}^{n_d} k^{(d)}_{j_d})!}{\Gamma(\beta (\sum_{j_1=1}^{n_1} k^{(1)}_{j_1}+...+\sum_{j_d=1}^{n_d} k^{(d)}_{j_d})+1)}\prod_{i=1}^d\frac{(-1/2)^{ n_i} (2n_i)!}{\prod_{j_i=1}^{n_i}k_{j_i}^{(i)}!{j_i}^{k_{j_i}^{(i)}}} \notag
\end{eqnarray}
and by $\sum_{\boldsymbol{k}}$ we denote the sum over all $n$-tuples of non-negative integers $\boldsymbol{k}=(k_1,...,k_n)$, such that $k_1+2k_2+...+nk_n=n$. 
\end{theorem}
\begin{proof}
In the even-order case, we can write that
\begin{eqnarray*}
 &&\mathbb{E[}a_{\ell _{1}m_{1}}^{2n_1}\cdots a_{\ell _{d}m_{d}}^{2n_d}]=\mathbb{E}\left[\mathcal{H}^{n_1}_{\ell_1}(T_\beta)\cdots \mathcal{H}^{n_d}_{\ell_d}(T_\beta)Z^{2n_1}_{\ell_1 m_1}\cdots Z^{2n_d}_{\ell_d m_d}\right]\notag \\
&=&(2n_1-1)!!\cdots (2n_d-1)!!C^{n_1}_{\ell_1}\cdots C^{n_d}_{\ell_d}\mathbb{E}\left[\mathcal{H}^{n_1}_{\ell_1}(T_\beta)\cdots \mathcal{H}^{n_d}_{\ell_d}(T_\beta)\right] \notag \\
&=&(2n_1-1)!!\cdots (2n_d-1)!!C^{n_1}_{\ell_1}\cdots C^{n_d}_{\ell_d}\int_0^\infty\mathbb{E}\mathcal{H}^{n_1}_{\ell_1}(t)\cdots \mathbb{E}\mathcal{H}^{n_d}_{\ell_d}(t)M_\beta(t)dt.
\end{eqnarray*}

The mixed moments of the subordinators can be evaluated, by the Fa\'a di Bruno formula, for any $t \geq 0,$ $n_i \in \mathbb{N}$, $i=1,...,d$:
     \begin{eqnarray}&&\mathbb{E}\left(\mathcal{H}_1^{n_1}(t)\cdots\mathcal{H}_d^{n_d}(t)\right)\notag\\
     &=&(-1)^{\sum_{i=1}^{d}n_i} \frac{\partial^{n_1+...+n_d}}{\partial\kappa_1^{n_1}\cdots\partial\kappa_d^{n_d}}e^{-t\sum_{i=1}^d\varphi_i(\kappa_i) }\bigg|_{\boldsymbol{\kappa}=\boldsymbol{0}} \label{momSub2} \\
     &=&(-1)^{\sum_{i=1}^{d}n_i}  e^{-t\sum_{i=1}^d\varphi_i(\kappa_i)}\prod_{i=1}^d\sum_{\boldsymbol{k}^{(i)} }
    \frac{n_i! }{k^{(i)}_1!\cdots k^{(i)}_{n_i}!
    }\prod_{j_i=1}^{n_i}\left( \frac{(-t)\varphi^{(j_i)}(\kappa_i)}{j_i!}\right)^{k_{j_i}}\bigg|_{\boldsymbol{\kappa}=\boldsymbol{0}} \notag \\
     &=&\prod_{i=1}^d \sum_{\boldsymbol{k}^{(i)}} \frac{n_i! }{k^{(i)}_1!\cdots k^{(i)}_{n_i}!
    }\prod_{j_i=1}^{n_i}\left( \frac{(-1)^{j_i-1}\varphi^{(j_i)}(0)}{j_i!}\right)^{k^{(i)}_{j_i}}t^{\sum_{j_i=1}^{n_i} k^{(i)}_{j_i}}. \notag 
     \end{eqnarray}
     The last step follows by recalling that $\sum_{j=1}^njk_j=n$ and clearly shows that the moments of any even order are non-negative, since $\varphi_i(\cdot)$ are a Bernstein functions.

      Therefore we get, by considering \eqref{mom},
      \begin{eqnarray}
&&\mathbb{E[}a_{\ell _{1}m_{1}}^{2n_1}\cdots a_{\ell _{d}m_{d}}^{2n_d}] \notag \\
&=&\prod_{i=1}^d\left[C^{n_i}_{\ell_i}\frac{(-1)^{n_i} (2n_i)!}{2^{n_i}}\right]
\sum_{\boldsymbol{k}^{(1)}}\cdots \sum_{\boldsymbol{k}^{(d)}}
\frac{( \sum_{j_1=1}^{n_1} k^{(1)}_{j_1}+...+\sum_{j_d=1}^{n_d} k^{(d)}_{j_d})!}{\Gamma(\beta ( \sum_{j_1=1}^{n_1} k^{(1)}_{j_1}+...+\sum_{j_d=1}^{n_d} k^{(d)}_{j_d})+1)}  \times \notag \\
&\times&
\prod_{j_1=1}^{n_1}\left[\frac{1}{k^{(1)}_{j_1}!}\left(- \frac{\varphi_{\ell_1}^{(j_1)}(0)}{j_1!}\right)^{k^{(1)}_{j_1}}\right]\cdots \prod_{j_d=1}^{n_d}\left[\frac{1}{k^{(d)}_{j_d}!}\left(- \frac{\varphi_{\ell_d}^{(j_d)}(0)}{j_d!}\right)^{k^{(d)}_{j_d}}\right], \notag 
\end{eqnarray}
so that \eqref{moments} follows.

In the odd-order case (i.e. for at least one of the orders equal to $r_i=2n_i-1)$, it is enough to recall \eqref{defa} and that the r.v.'s $Z_{\ell_i m_i}$ are centered Gaussian (for any $i$), so that the conclusion holds, by considering again \eqref{momSub2} and that the condition $\varphi_{\ell_i}^{(j)}< \infty,$ for $j=1,...n_i$, is enough to ensure that $\mathbb{E}(\mathcal{H}^{r_i/2}_{\ell_i}(T_\beta))< \infty$.

\end{proof}

\bigskip

We will show now that, while in the global variance case non-Gaussianity persists at all frequencies, in the random angular power spectrum case, considered here, high-frequency limits depend on the asymptotic behaviour of the $\ell$-th subordinator component, i.e. $\mathcal H_\ell(T_\beta)$.
\begin{theorem}\label{theoconv}
Let, for any $\ell$, $a_{\ell m}=\sqrt{\mathcal H_\ell(T_\beta)}\,Z_{\ell m}$,
for $Z_{\ell m}\sim N(0,C_\ell),$
 independent of \(\mathcal H_\ell(T_\beta)\). Assume that, for some
positive sequence \(\{r_\ell\}_{\ell}\) and for the Laplace exponent $\varphi_\ell(\cdot)$ of the $\ell$-th subordinator, the following convergence holds
\begin{equation}
\lim_{\ell \to\infty}\varphi_\ell(s/r_\ell)=\psi(s), \qquad s >0,\label{convfi}
\end{equation}
for a Bernstein function $\psi:\mathbb{R}_+\to\mathbb R_+.$
 Then
\begin{equation}
    \frac{a_{\ell m}}{\sqrt{C_\ell r_\ell}}
\overset{d}{\Longrightarrow}
\sqrt V\,N, \qquad \ell \to \infty,
\label{convfi2}
\end{equation}
where $\overset{d}{\Longrightarrow}$ denotes convergence in distribution, \(V\) is a non-negative random variable, independent of \(N\sim N(0,1)\), with $\mathbb Ee^{-sV}=E_\beta(-\psi(s))$.
\end{theorem}
\begin{proof}
As a consequence of the assumption \eqref{convfi}, we have that
\begin{equation}\lim_{\ell \to\infty}E_\beta(-\varphi_\ell(s/r_\ell))=:E_\beta(-\psi(s)), \qquad s>0.\label{convfi3}\end{equation}
The function $E_\beta(-\psi(\cdot))$ is the composition of a completely monotone and a Bernstein function and then it is itself completely monotone (see \cite{SCH}, Theorem 3.7, p.27).
Then, by the continuity theorem for Laplace transforms, the convergence \eqref{convfi3}
is enough to get
\[
\frac{\mathcal H_\ell(T_\beta)}{r_\ell}\overset{d}{\Longrightarrow} V,
\qquad \ell\to\infty.
\]
For fixed \(\ell\) and \(m\), by the representation of the spherical harmonic
coefficients, we have that
\[
\frac{a_{\ell m}}{\sqrt{C_\ell r_\ell}}
=
\sqrt{\frac{\mathcal H_\ell(T_\beta)}{r_\ell}}\,N_{\ell m},
\]
where $N_{\ell m}\sim N(0,1)$, for every \(\ell\). Since \(N_{\ell m}\) is independent of
\(\mathcal H_\ell(T_\beta)\), we also have that
\[
\left(
\frac{\mathcal H_\ell(T_\beta)}{r_\ell},
N_{\ell m}
\right)
\overset{d}{\Longrightarrow} 
(V,N),\qquad \ell \to \infty,
\]
where \(N\) is independent of \(V\).
By the continuous mapping theorem, applied to the map
$(x,y)\mapsto \sqrt{x}\,y,$
for $x\ge0,\ y\in\mathbb R,$
we obtain
\[
\sqrt{\frac{\mathcal H_\ell(T_\beta)}{r_\ell}}\,N_{\ell m}
\overset{d}{\Longrightarrow} 
\sqrt V\,N,\qquad \ell \to \infty
\]
which gives the final result.
\end{proof}

\bigskip

We describe in Appendix A some relevant special cases of the GSRF, obtained by specifying the multivariate subordinator with independent components: the stable case with variable order, as well as the compound Poisson, and the Gamma cases. In these particular settings, it is possible to derive explicit forms of the results given in Theorems \ref{mixmom} and \ref{theoconv}.

\section{GSRFs with dependent subordinators}\label{secdepsub}
If we remove the assumption of independence among subordinator's components, i.e. the Laplace exponent of the multivariate subordinator $\Psi_\beta(\cdot,...,\cdot)$ is in the general form given in Def. \ref{defsub}, 
then, instead of Def. \ref{defGRFeq}, we have

\begin{definition}\label{defgen}
A spherical random field with power spectrum $\left\{ C_{\ell
}\right\} $ is Grey if, for any $d \in \mathbb N$ and any choice of multipoles $\ell_1,...,\ell_d$, we have that%
\begin{equation}
\Phi _{\boldsymbol{a}_{\ell _{1\cdot}},...,\boldsymbol{a}_{\ell
_{d\cdot}}}(\boldsymbol{\xi}_{\ell_1\cdot},\boldsymbol{\xi}_{\ell_2\cdot},...,\boldsymbol{\xi}_{\ell_d\cdot})=E_\beta\left(-\varphi_{\ell_1,...,\ell_d}(C_{\ell_1}
\Vert\xi_{\ell_1\cdot}\Vert^2/2,...,C_{\ell_d}
\Vert\xi_{\ell_d\cdot}\Vert^2/2)\right),
\end{equation}%
for $\beta \in (0,1)$, $\xi_i \in\mathbb{R}$, $i=1,...,d$.
\end{definition}

The representation
of the random field as \begin{equation} \mathcal T(x) =\sum_{\ell\geq 0} \sum_{m=-\ell}^{\ell} a_{\ell,m}Y_{\ell m}(x)= \sum_{\ell\geq 0} \sum_{m=-\ell}^{\ell} \sqrt{ \mathcal H_\ell(T_\beta) }\, Z_{\ell m}Y_{\ell m}(x),\qquad x\in\mathbb S^2,\notag \end{equation}
where now \(\mathcal H_\ell\) is the $\ell$-th component of a general multivariate subordinator, admits a natural microscopic interpretation.

The quantities \(\{\mathcal H_\ell(T_\beta)\}_{\ell\ge0}\) can be viewed as
random energy levels associated with the different harmonic scales of the
field. In many physical systems, such energies arise as macroscopic limits of
the cumulative effect of a large number of small random transfers occurring
within each spectral band.

More precisely, for each multipole \(\ell\), consider a sequence of cumulative jump processes
\(\{S_\ell^{(n)}(t)\}_{t\ge0}\), describing the random accumulation of energy within the corresponding spectral band. The jumps are assumed to belong to the domain of attraction of the \(\ell\)-th component of a multivariate subordinator
\(\{\mathcal H_\ell(t)\}_{t\ge0}\), so that, under a suitable scaling, the processes
\(S_\ell^{(n)}\) converge to \(\mathcal H_\ell\). The L\'evy measure of the limiting multivariate subordinator encodes the statistics of the microscopic energy-transfer mechanism as well as the dependence structure among different multipoles.

The occurrence of energy-transfer events is governed by a renewal process with heavy-tailed interarrival times. Consequently, the associated counting process converges, after rescaling, to the inverse of a \(\beta\)-stable subordinator. Therefore, the random operational time \( T_\beta\), whose distribution is the Mittag-Leffler law, arises naturally as the scaling limit of the underlying renewal dynamics.

In this framework, as we will see in the following subsection, Grey spherical random fields can be interpreted as macroscopic scaling limits of random spectral energy cascades evolving on multiple harmonic scales, where the non-Gaussian behaviour results from the combined effect of the multivariate subordinator governing the energy transfers and the anomalous temporal dynamics generated by heavy-tailed waiting times.

\subsection{The Grey spherical random fields as scaling limits of CTRW's}

In the following theorem we will say that, for any $n \geq 1$ and for every finite collection
\(\ell_1,\ldots,\ell_d\), the sequence of i.i.d. random vectors in \(\mathbb R_+^d\)
\begin{equation}
\boldsymbol J_j^{(n)}
=
\bigl(
J_{\ell_1,j}^{(n)},\ldots,
J_{\ell_d,j}^{(n)}
\bigr),
\qquad j\ge1,\label{trar}
\end{equation}
belongs to the domain of attraction of
\(
\boldsymbol{\mathcal{H}}(1)
=
\bigl(
\mathcal H_{\ell_1}(1),\ldots,
\mathcal H_{\ell_d}(1)
\bigr)
,\)
where 
$\{\boldsymbol{\mathcal{H}}(t)\}_{t \geq 0}$ is the multivariate subordinator given in Definition \ref{defsub} and we will write $\boldsymbol J^{(n)}\in\mathcal H$-DOA, if the following conditions hold:
\begin{equation}
n\,\mathbb P(\boldsymbol J_1^{(n)}\in\cdot)
\overset{v}{\Longrightarrow}
\nu_{\ell_1,\ldots,\ell_d}(\cdot), \qquad n \to \infty,\label{cond1}
\end{equation}
on
\(\mathbb R_+^d\setminus\{\boldsymbol0\}\), where
\(\nu_{\ell_1,\ldots,\ell_d}\) is the Lévy measure of
\((\mathcal H_{\ell_1}(1),\ldots,\mathcal H_{\ell_d}(1))\), and

\begin{equation}
\lim_{\varepsilon\downarrow0}
\lim_{n\to\infty}
n\,
\mathbb E
\!\left[
\boldsymbol J_1^{(n)}
\boldsymbol1_{\{\|\boldsymbol J_1^{(n)}\|\le\varepsilon\}}
\right]
=
\boldsymbol b_{\ell_1,\ldots,\ell_d},\label{cond2}
\end{equation}
where
\(\boldsymbol b_{\ell_1,\ldots,\ell_d}\)
is the drift vector of the subordinator.

The notation
\(\overset{v}{\Longrightarrow}\)
denotes vague convergence of Radon measures, namely,
\[
\mu_n
\overset{v}{\Longrightarrow}
\mu
\]
if and only if
\[
\int_{\mathbb R_+^d\setminus\{\boldsymbol0\}}
f(x)\,\mu_n(dx)
\longrightarrow
\int_{\mathbb R_+^d\setminus\{\boldsymbol0\}}
f(x)\,\mu(dx)
\]
for every
\(f\in C_c(\mathbb R_+^d\setminus\{\boldsymbol0\})\),
where
\(C_c(\mathbb R_+^d\setminus\{\boldsymbol0\})\)
denotes the space of continuous functions with compact support. Moreover, note that the conditions \eqref{cond1} and \eqref{cond2} are the specialization to multivariate subordinators of the classical convergence criteria for infinitesimal triangular arrays towards infinitely divisible laws; see, e.g., \cite{KAL}, ch. 5, and \cite{JAC}, ch. VII, for the corresponding distributional and functional convergence results.

Let us also consider the fractional Poisson process $\{N_\beta(t)\}_{t \geq 0}$, i.e. a renewal process with (i.i.d.) power-law inter-arrival times $U_j$ with $\mathbb P(U_j>t)=E_\beta(-t^\beta)$, for $\beta \in (0,1)$ and $j=1,2...$. Then, it is well-known that the following convergence holds \begin{equation}\{c^{-\beta}N_\beta(\lfloor ct \rfloor)\}_{t \geq 0} \overset{J_1}{\Longrightarrow}\{\mathcal L_\beta(t)\}_{t \geq 0}, \qquad c \to \infty,\label{nbeta}\end{equation} in $D([0,\infty))$ with the $J_1$ topology (see \cite{MEE}, p.104), where $\{\mathcal{L}_\beta(t)\}_{t \geq 0}$ is the inverse of the $\beta$-stable subordinator.  Note that \(T_\beta\overset{\text{d}}{=}\mathcal L_\beta(1) \), since $\mathbb P(\mathcal{L}_\beta(t) \in dx)/dx=t^{-\beta}M_\beta(x t^{-\beta})$ and $M_\beta(\cdot)$ is the function given in \eqref{mai}.

We start by proving the convergence, in the sense of f.d.d.'s, to the truncated GSRF defined as
\begin{equation} \mathcal T_L(x): = \sum_{\ell= 0}^L \sum_{m=-\ell}^{\ell} \sqrt{ \mathcal H_\ell(T_\beta) }\, Z_{\ell m}Y_{\ell m}(x),\qquad x\in\mathbb S^2,\label{TL} \end{equation}
for every fixed $L$.
\begin{theorem}\label{teoconv}
Let \(\{Z_{\ell m}\}\) be independent Gaussian random variables with
\(Z_{\ell m}\sim N(0,C_\ell)\), 
let
$\boldsymbol J^{(n)}\in\mathcal H$-DOA, and let, for every \(n\ge1\), \(\ell \in\mathbb R_+\) and \(t\ge0\),
\[
S^{(n)}_\ell(t)
:=
\sum_{j=1}^{N_\beta({\lfloor n^{1/\beta}t\rfloor)}}J_{\ell,j}^{(n)}, 
\]
where, by assumption, $\boldsymbol J^{(n)}$ is independent of the process $N_\beta$ and of \(\{Z_{\ell m}\}\), for any $n \in \mathbb{N}.$
Then, for every fixed \(L\), \begin{equation} \mathcal{T}_L^{(n)}(x): = \sum_{\ell=0}^{L} \sum_{m=-\ell}^{\ell} \sqrt{ S_\ell^{(n)}(1) }\, Z_{\ell m}Y_{\ell m}(x) \Longrightarrow\mathcal{T}_L(x),\qquad n \to \infty,\label{trunca} \end{equation}
where $\Longrightarrow$ denotes weak convergence of probability measures in \(L^2(\mathbb S^2)\).
\end{theorem}
\begin{proof}
Let us define
\[
\boldsymbol R^{(n)}(u)
=
\sum_{j=1}^{[nu]}\boldsymbol J_j^{(n)},
\qquad u\geq 0,
\]
for the sequence of random vectors given in \eqref{trar}. 
Since, by assumption, the latter belongs to the domain of attraction of the
multivariate subordinator, we have that
\[
\boldsymbol R^{(n)}(\cdot)
\overset{J_1}{\Longrightarrow}
\mathcal H(\cdot), \qquad n \to \infty,
\]
in
\(
D([0,\infty),\mathbb R_+^d).
\) Moreover, by recalling \eqref{nbeta}, for the fractional Poisson process \(N_\beta\), the scaling limit
\[
n^{-1}N_\beta(\lfloor n^{1/\beta}\cdot\rfloor)
\overset{J_1}{\Longrightarrow}
\mathcal L_\beta(\cdot), \qquad n \to \infty,
\]
holds in \(D([0,\infty),\mathbb R_+)\).
Since the jump array is independent of \(N_\beta\), the two convergences hold
jointly:
\[
\left(
\boldsymbol R^{(n)}(\cdot),
n^{-1}N_\beta(\lfloor n^{1/\beta}\cdot\rfloor)
\right)
\overset{J_1}{\Longrightarrow}
\left(
\mathcal H(\cdot),
\mathcal L_\beta(\cdot)
\right), \qquad n \to \infty.
\]
Now observe that
\[
\bigl(
S_{\ell_1}^{(n)}(t),\ldots,S_{\ell_d}^{(n)}(t)
\bigr)
=
\boldsymbol R^{(n)}
\left(
n^{-1}N_\beta(\lfloor n^{1/\beta}t\rfloor)
\right).
\]
Therefore, by the continuous mapping theorem for random time changes,
\[
\bigl(
S_{\ell_1}^{(n)}(\cdot),\ldots,S_{\ell_d}^{(n)}(\cdot)
\bigr)
\overset{J_1}{\Longrightarrow}
\bigl(
\mathcal H_{\ell_1}(\mathcal L_\beta(\cdot)),
\ldots,
\mathcal H_{\ell_d}(\mathcal L_\beta(\cdot))
\bigr).
\]
In particular, at \(t=1\),
\[
\bigl(
S_{\ell_1}^{(n)}(1),\ldots,S_{\ell_d}^{(n)}(1)
\bigr)
\overset{f.d.d.}{\Longrightarrow}
\bigl(
\mathcal H_{\ell_1}(T_\beta),
\ldots,
\mathcal H_{\ell_d}(T_\beta)
\bigr).
\]
Since the Gaussian variables \(\{Z_{\ell m}\}\) are independent of the jump
processes and of \(N_\beta\), we also have joint convergence
\[
\left(
S_{\ell_i}^{(n)}(1),
Z_{\ell_i m_i}
\right)_{i=1}^d
\overset{\text{d}}{\Longrightarrow}
\left(
\mathcal H_{\ell_i}(T_\beta),
Z_{\ell_i m_i}
\right)_{i=1}^d,\qquad n \to \infty .
\]
By applying again the continuous
mapping theorem,
\[
\left(
\sqrt{S_{\ell_i}^{(n)}(1)}\,Z_{\ell_i m_i}
\right)_{i=1}^d
\overset{\text{d}}{\Longrightarrow}
\left(
\sqrt{\mathcal H_{\ell_i}(T_\beta)}\,Z_{\ell_i m_i}
\right)_{i=1}^d, \qquad n \to \infty .
\]
Finally, for fixed \(L\), the vector
\(
\bigl(
\mathcal{T}_L^{(n)}(x_1),\ldots,\mathcal{T}_L^{(n)}(x_k)
\bigr)
\)
is a finite linear combination of the random variables
\[
\sqrt{S_\ell^{(n)}(1)}\,Z_{\ell m},
\qquad
0\leq \ell\leq L,\quad -\ell\leq m\leq \ell.
\]
Therefore a final application of the continuous mapping theorem gives
\[
\bigl(
\mathcal{T}_L^{(n)}(x_1),\ldots,\mathcal{T}_L^{(n)}(x_k)
\bigr)
\overset{f.d.d.}{\Longrightarrow}
\bigl(
\mathcal{T}_L(x_1),\ldots,\mathcal{T}_L(x_k)
\bigr), \qquad n \to \infty,
\]
for every \(x_1,\ldots,x_k\in\mathbb S^2\).
Since we are dealing with fixed $L$, convergence in $L^2(\mathbb S^2)$ follows immediately.. 
\end{proof}

We now prove the full convergence to the field \[
\mathcal{T}(x)
=
\sum_{\ell\ge0}\sum_{m=-\ell}^{\ell}
\sqrt{H_\ell(T_\beta)}\,Z_{\ell m}Y_{\ell m}(x).
\] in \(L^2(\mathbb S^2)\).

\begin{theorem}
Let the assumptions of Theorem \ref{teoconv} and \eqref{sqrint} hold; suppose moreover that
\begin{equation}
\sum_{\ell\ge0}(2\ell+1)C_\ell
\sup_{n\ge1}\mathbb E\left[S_\ell^{(n)}(1)\right]
<\infty .\label{hyp:sum}
\end{equation}
and define
\[
\mathcal T^{(n)}(x):
=
\sum_{\ell\ge0}\sum_{m=-\ell}^{\ell}
\sqrt{S_\ell^{(n)}(1)}\,Z_{\ell m}Y_{\ell m}(x),
\qquad x\in\mathbb S^2.
\]
Then
\begin{equation}
\mathcal T^{(n)}\Longrightarrow \mathcal T, \qquad n \to \infty,\label{convbill}
\end{equation}
in \(L^2(\mathbb S^2)\).

\end{theorem}

\begin{proof}
By orthonormality of the spherical harmonics and independence of
$S_\ell^{(n)}(1)$ and $Z_{\ell m}$,
\[
\begin{aligned}
\sup_{n\geq1}
\mathbb E\left[
\|\mathcal T^{(n)}-\mathcal T_L^{(n)}\|_{L^2(\mathbb S^2)}^2
\right]
&\leq
\sum_{\ell>L}(2\ell+1)C_\ell
\sup_{n\geq1}\mathbb E[S_\ell^{(n)}(1)],
\end{aligned}
\]
for any fixed \(L.\)
Hence, by the summability assumption \eqref{hyp:sum},
\begin{equation}
\lim_{L\to\infty}\sup_{n\geq1}
\mathbb E\left[
\|\mathcal T^{(n)}-\mathcal T_L^{(n)}\|_{L^2(\mathbb S^2)}^2
\right]
=0.\label{bill}
\end{equation}

Moreover, by assumption \eqref{sqrint},
\[
\mathbb E\left[
\|\mathcal T-\mathcal T_L\|_{L^2(\mathbb S^2)}^2
\right]
=
\sum_{\ell>L}(2\ell+1)C_\ell
\mathbb E[H_\ell(T_\beta)]
\longrightarrow0,
\qquad L\to\infty,
\]
so that
\begin{equation}
\mathcal T_L\longrightarrow\mathcal T
\quad\text{in mean-square in }L^2(\mathbb S^2), \qquad L \to \infty.\label{convprob}
\end{equation}
Recall that by Theorem \ref{teoconv}, for every fixed $L$,
\begin{equation}
\mathcal T_L^{(n)}
\Longrightarrow
\mathcal T_L,
\qquad n\to\infty,\label{fdd}
\end{equation}
as random elements of $L^2(\mathbb S^2)$.
The convergence in \eqref{convbill} follows by considering \eqref{bill}, \eqref{convprob} and \eqref{fdd}. 
\end{proof}

\subsection{The case of multivariate tempered stable subordinators}

In order to obtain some explicit results on moments and cumulants, we will consider the wide class of multivariate tempered stable subordinators (hereafter TSS) considered in \cite{SEM} (see also \cite{PER}), which, in the special case of independence, includes many of the examples considered in Appendix A.

Let $\alpha \in [0,1),$ let $\lambda(\cdot)$ be a finite measure on $\mathbb{S}_+^{d-1}:=\left\{\boldsymbol{x}\in \mathbb{S}^{d-1}:x_i\geq 0, i=1,...,d-1\right\}$  and $b:\mathbb{S}_+^{d-1}\to (0,+\infty)$ a Borel-measurable function; then, in this case, the Laplace exponent is equal to
\[
\varphi(\boldsymbol{\kappa})=\int_{\mathbb{S}_+^{d-1}}\int_{\mathbb{R}_+}(1-e^{-r\langle\boldsymbol{\kappa},\boldsymbol{x}\rangle})\frac{e^{-b(\boldsymbol{x}) r}}{r^{\alpha+1}}dr\lambda(d\boldsymbol{x})
\]
and the 
L\'evy measure is 
\[
\nu(dr,d\boldsymbol{x})=\frac{e^{-b(\boldsymbol{x}) r}}{r^{\alpha+1}}dr\lambda(d\boldsymbol{x}),
\]
satisfying the existence conditions that $\int_{\Vert \boldsymbol{x}\Vert \leq 1}\Vert \boldsymbol{x}\Vert\nu(d\boldsymbol{x}) <\infty$ and $\int_{\mathbb{R}_+^d}\nu(d\boldsymbol{x})=\infty$ (see Prop. 2.1 in \cite{SEM}).
It is proved in \cite{SEM} that the $k$-th moment is always finite, for $k\leq \alpha$, while, for $k>\alpha$, it is finite if and only if $\int_{\mathbb{S}_+^{d-1}}b(\boldsymbol{v})^{-\alpha-k}\lambda(d\boldsymbol{v})<\infty$. Moreover, if we denote the cumulant generating function $\Xi(\boldsymbol{\theta}):=\log\mathbb Ee^{\sum_{j=1}^d\theta_j\mathcal{H}_j(t)}$ and the mixed cumulants as $\mu_{1...d}:=\frac{\partial^{d}}{\partial \theta_1...\partial \theta_d}\Xi(\boldsymbol{\theta})\vert_{\boldsymbol{\theta}=\boldsymbol{0}}$, then 
the latter are equal to
\begin{equation}
\mu_{1...d}
=t\Gamma(d-\alpha)\int_{\mathbb{S}_+^{d-1}}b(\boldsymbol{v})^{\alpha-d}\boldsymbol{v}_1\cdots \boldsymbol{v}_d\lambda(d\boldsymbol{v}). \label{tss}
\end{equation}
As a consequence, it is easy to see that,
for example,
\begin{eqnarray}
\mathbb{E}(a^2_{\ell_1 m_1}a^2_{\ell_2 m_2})&=&C_{\ell_1}C_{\ell_2}\mathbb{E}\left[\mathcal{H}_{\ell_1}(T_\beta)\mathcal{H}_{\ell_2}(T_\beta) \right]\label{gen}\\
&=&C_{\ell_1}C_{\ell_2}\frac{\Gamma(2-\alpha)}{\Gamma(\beta+1)}\int_{\mathbb{S}_+^{d-1}}b(\boldsymbol{v})^{\alpha-2}\boldsymbol{v}_{\ell_1} \boldsymbol{v}_{\ell_2}\lambda(d\boldsymbol{v})+  \notag\\
&+&C_{\ell_1}C_{\ell_2}\frac{\Gamma(1-\alpha)^2}{\Gamma(\beta+1)^2}\int_{\mathbb{S}_+^{d-1}}b(\boldsymbol{v})^{\alpha-1}\boldsymbol{v}_{\ell_1} \lambda(d\boldsymbol{v})\int_{\mathbb{S}_+^{d-1}}b(\boldsymbol{v})^{\alpha-1}\boldsymbol{v}_{\ell_2}\lambda(d\boldsymbol{v}).
 \notag   \end{eqnarray}

    \begin{remark}
Note that we can obtain, as special cases of the TSS, some examples of the subordinators with independent components (which were introduced in the previous section and extensively treated in Appendix A), by choosing $\nu(d\boldsymbol{x})=\sum_{i=1}^dc_idx_i\delta_0(d\overline{\boldsymbol{x}}_i)$, where $c_i>0,$ $i=1,...d$, and we denoted by $\delta_0(\cdot)$ the Dirac delta function and by $ \overline{\boldsymbol{x}}_i$ the $1\times (d-1)$vector $\boldsymbol{x}$ where the $i$-th component is missing. 

In particular, the multivariate Gamma subordinator (for $\alpha=0$, $b(\boldsymbol{x})\equiv b$ and $c_i=\gamma_i$), the multivariate inverse Gaussian subordinator (for $\alpha=1/2$, $b(\boldsymbol{x})\equiv\gamma^2$ and $c_i=1$, $i=1,..,d$) and the multivariate stable subordinator (for $\alpha \in (0,1)$ and $b(\boldsymbol{x})\equiv 0$).

Indeed, under these special choices, formula \eqref{gen} coincides with \eqref{stir} and \eqref{invgau}, respectively; on the other hand, in the stable case, we can see directly from \eqref{gen} that the second mixed moments are infinite, for $b(\boldsymbol{v})\equiv 0.$
    \end{remark}

We now study the asymptotic behaviour of cross-cumulants.

An important issue concerns the persistence of dependence across different
multipoles. In the framework of Definition \ref{defgen}, this dependence is entirely
generated by the dependence structure of the underlying multivariate
subordinator.

To quantify this effect, consider the mixed fourth-order cumulant, which can be written, by recalling \eqref{gen}, as
\begin{eqnarray}
\operatorname{cum}
\left(
a_{\ell_1m_1}^2,
a_{\ell_2m_2}^2
\right)
&=&
E\!\left[
a_{\ell_1m_1}^2
a_{\ell_2m_2}^2
\right]
-
E\!\left[
a_{\ell_1m_1}^2
\right]
E\!\left[
a_{\ell_2m_2}^2
\right] \notag\\
&=&
C_{\ell_1}C_{\ell_2}
\frac{\Gamma(2-\alpha)}
{\Gamma(\beta+1)}
\int_{S_+^{d-1}}
b(v)^{\alpha-2}
v_{\ell_1}v_{\ell_2}
\,\lambda(dv). \notag
\end{eqnarray}

This expression shows that the dependence between different multipoles is
entirely determined by the spectral measure \(\lambda\) of the multivariate
tempered stable subordinator.

To investigate asymptotic dependence across multipoles, we extend the
construction to a countable family of dependent subordinators
\(\{\mathcal H_\ell(t)\}_{\ell\ge1}\), whose finite-dimensional distributions are
consistent with Definition \ref{defgen}.

In the infinite-dimensional extension, let the spectral measure
\(\lambda\) be defined on
\[
S_+^\infty
=
\left\{
v\in \ell^2:\ v_\ell\ge0,\ \sum_{\ell\ge1}v_\ell^2=1
\right\}
\]
and let $\lambda(\mathbb S^\infty_+)<\infty$; if we also assume that $b(v)\geq b_0>0$, for any $v$, then
\[
\int_{S_+^\infty} b(v)^{\alpha-2}\lambda(dv)<\infty .
\]
Under these assumptions, since \(v_\ell\to0\) for every \(v\in\ell^2\), by dominated convergence,
we have
\[
\int_{S_+^\infty}
b(v)^{\alpha-2}v_{\ell_1}v_{\ell_2}\lambda(dv)
\longrightarrow 0,
\qquad |\ell_1-\ell_2|\to\infty .
\]

As a consequence,
\[
\operatorname{cum}
\left(
a_{\ell_1m_1}^2,
a_{\ell_2m_2}^2
\right)
\longrightarrow 0,
\qquad
|\ell_1-\ell_2|\to\infty ,
\]
showing that the quadratic dependence between distant multipoles vanishes
asymptotically. On the other hand, if
\[
\limsup_{|\ell_1-\ell_2|\to\infty}
\left|
\int_{S_+^{d-1}}
b(v)^{\alpha-2}
v_{\ell_1}v_{\ell_2}
\,\lambda(dv)
\right|
>0,
\]
then the mixed cumulants remain bounded away from zero and the model
exhibits persistent dependence across harmonic scales.

Therefore, the asymptotic dependence structure of the GSRF is completely governed by the spectral measure of the underlying
multivariate tempered stable subordinator. In particular, the decay or
persistence of cross-cumulants is determined by the mixed angular weights
\[
\int_{S_+^{d-1}}
b(v)^{\alpha-2}
v_{\ell_1}v_{\ell_2}
\,\lambda(dv).
\]

\section{Some explicit bounds on the distance to Gaussianity}\label{sec:bounds}

As we discussed earlier in the paper, a key motivation of our work is the attempt to create a general class of random fields which embed the Gaussian case as a special one. It is hence natural to ask what is the distance, in a suitable probability metric, between our constructions and Gaussianity, as a function of the hyperparameters of the model. More precisely, we identify two main paths by which non-Gaussianity is introduced in the grey case, as follows:

\begin{itemize}

\item The Mittag-Leffler characteristic function returns Gaussianity only in the special case where the parameter $\beta$ is exactly equal to unity and the internal function $\varphi(\cdot)$ is the identity. As we depart from these special conditions, non-Gaussianity steps in;

\item While (under isotropy) Gaussianity entails independence of the spherical harmonic coefficients, in the grey case correlation among the coefficients can be introduced.
\end{itemize}

Our purpose in this Section is to derive explicit quantitative bounds between the distribution of the Grey spherical random fields and the Gaussian case, as a function of three parameters:

\begin{itemize}

\item the distance from $\beta$ to one

\item the distance from $\varphi(\cdot)$ to the identity function, i.e., the distance from the quadratic form in the squared spherical harmonic coefficients and the simpler quadratic norm of the latter
   
\end{itemize}

We start with the following preliminary results, which will be useful in order to prove uniform upper bounds for the distance between the grey and the Gaussian characteristic functions.

\begin{lemma}\label{lemsup}
Let \(0<\beta<1\). Then, as \(\beta\uparrow 1\),
\[
\sup_{z>0}\left|E_\beta(-z)-e^{-z}\right|
=
C_*(1-\beta)+o(1-\beta),
\]
where
\begin{equation}
C_*
=
\left|
\frac{1-(z_*+1)e^{-z_*}}{z_*-1}
\right|,\label{cstar}
\end{equation}
and \(z_*>1\) is the unique solution of
\begin{equation}
(1-z)\bigl(\operatorname{Ei}(z)-\log z\bigr)=2-e^z. \label{zstar}
\end{equation}
\end{lemma}

\begin{proof}
Let us denote by $\gamma$ the Eulero-Mascheroni constant and by $\psi(\cdot)$ the digamma function (i.e. \(\psi(z):=\Gamma'(z)/\Gamma(z)\), $z \in \mathbb C$); then the Mittag-Leffler function \eqref{ML} can be differentiated term by term, with respect to \(\beta\), as
\[
\left.\frac{\partial}{\partial\beta}E_\beta(-z)\right|_{\beta=1}
=
-\sum_{n=1}^{\infty}
\frac{n\psi(n+1)}{n!}(-z)^n=
-\sum_{n=1}^{\infty}
\frac{n(H_n-\gamma)}{n!}(-z)^n=-\gamma z e^{-z}+z\frac{d}{dx}\left. \sum_{n=1}^{\infty}
\frac{H_n}{n!}x^n\right|_{x=-z},
\]
by recalling that, for any $n \in \mathbb N$,
\[
\psi(n+1)=\sum_{k=1}^{n}\frac{1}{k}-\gamma=:H_n-\gamma
\]
(see \cite{WIT}, sect. 12.3).
We now compute this series explicitly, by using the classical generating function of the harmonic numbers, i.e.,
\[
\sum_{n=1}^{\infty}H_n\frac{x^n}{n!}
=
e^x\operatorname{Ein}(x),
\]
where
\(
\operatorname{Ein}(x)
:=
\sum_{n=1}^{\infty}
(-1)^{n-1}x^n/n\,n!.
\)
We may now set $x=-z$, since $\operatorname{Ein}$ is entire, and we recall that
\begin{eqnarray}
    \operatorname{Ein}(-z)
&=&
\gamma+\log z-\operatorname{Ei}(z),\qquad z>0,\notag\\
\operatorname{Ein}'(z)&=&\frac{1-e^{-z}}{z},\qquad z \neq 0,
\notag
\end{eqnarray}
where $\operatorname{Ei}(\cdot)$ is the classical exponential integral, i.e., $\operatorname{Ei}(x):=\int_{-\infty}^xe^{-t}/t\,dt$. Therefore we get
\begin{equation}
\left.\frac{\partial}{\partial\beta}E_\beta(-z)\right|_{\beta=1}
=
1-e^{-z}
-
ze^{-z}\bigl(\operatorname{Ei}(z)-\log z\bigr)=:F(z),
\qquad z>0,\label{eq:Fdef}
\end{equation}
so that
\begin{equation}
E_\beta(-z)-e^{-z}
=
(\beta-1)F(z)+o(|\beta-1|)\label{asyML}
\end{equation}
locally uniformly in \(z\). In order to pass from local uniform convergence to uniform convergence on the whole half-line, we apply the uniform asymptotic estimate, valid for \(\beta\) in a neighbourhood of \(1\),
\[
E_\beta(-z)
=
\frac{1}{z\Gamma(1-\beta)}
+
O\left(\frac{1}{z^2}\right),
\qquad z\to+\infty,
\]
uniformly in \(\beta\). Since
\[
\frac{1}{\Gamma(1-\beta)}
=
1-\beta+O\bigl((1-\beta)^2\bigr),
\]
we get, uniformly for large \(z\),
\[
E_\beta(-z)
=
\frac{1-\beta}{z}
+
O\bigl((1-\beta)^2 z^{-1}\bigr)
+
O(z^{-2}).
\]
On the other hand, 
\(
e^{-z}=o(z^{-2}),
\)
and
\(
F(z)\sim -1/z
\), \( z\to+\infty.
\)
Therefore \eqref{asyML} holds
uniformly also for \(z\ge R\), after choosing \(R\) large enough and then letting \(\beta\uparrow1\). 
It follows that
\[
\sup_{z>0}|E_\beta(-z)-e^{-z}|
=
|\beta-1|\sup_{z>0}|F(z)|
+
o(|\beta-1|)=
(1-\beta)C_*+o(1-\beta),
\]
since \(\beta<1\), with
\(
C_*=\sup_{z>0}|F(z)|.
\) Differentiating \(F(\cdot)\)
and using
\[
\frac{d}{dz}\bigl(\operatorname{Ei}(z)-\log z\bigr)
=
\frac{e^z-1}{z},
\]
we obtain
\[
F'(z)
=
2e^{-z}-1
-
(1-z)e^{-z}
\bigl(\operatorname{Ei}(z)-\log z\bigr).
\]
Thus the maximum of \(|F(\cdot)|\) is attained at the unique solution \(z_*>1\) of equation \eqref{zstar} and $C_*$ is given in \eqref{cstar}.
\end{proof}

\begin{remark}
The exact value of $C_*$, given in \eqref{cstar}, can be easily obtained by numerical methods and reads
\[
C_*\simeq 0.371856, \qquad \text{for } z_*\simeq 3.244025.
\]

\end{remark}

We now 
consider the case of a multivariate subordinator, under the assumption of finiteness of moments. In order to study the distance to Gaussianity of our Grey spherical fields, we will write down a rather general form of the function $\Phi(\cdot)$, and we will argue that it can be interpreted as a perturbation of a baseline standard Gaussian vector. More precisely, let us introduce the following assumption.

\begin{condition}\label{Luglio}
For every finite collection
\(\ell_1,\ldots,\ell_d\) and for \( t\geq0
\), let
\(
\boldsymbol{\mathcal H}_{\ell_1,\ldots,\ell_d}^{(\gamma)}(t)
:=
\bigl(
\mathcal H_{\ell_1}^{(\gamma)}(t),\ldots,
\mathcal H_{\ell_d}^{(\gamma)}(t)
\bigr)\),
be a \(d\)-dimensional subordinator with Laplace exponent
\[
\varphi_{\ell_1,\ldots,\ell_d}^{(\gamma)}
(\boldsymbol z)
:=
(1-\gamma)
\|\boldsymbol z\|_1
+
\gamma
\int_{\mathbb R_+^d\setminus\{\boldsymbol0\}}
\left(
1-e^{-\langle\boldsymbol z,\boldsymbol x\rangle}
\right)
\widehat\nu_{\ell_1,\ldots,\ell_d}
(d\boldsymbol x),
\]
where \(\|\boldsymbol z\|_1
:=
\sum_{j=1}^dz_j\), for \(\boldsymbol z\in\mathbb R_+^d\), and \(\gamma:=\gamma_{\ell_1,\ldots,\ell_d}\in [0,1]
\).
Assume that the L\'evy measure \(\widehat \nu_{\ell_1,..,\ell_d}(\cdot)\) is such that
\[
\int_{\mathbb R_+^d\setminus\{\boldsymbol0\}}
x_j\,
\widehat\nu_{\ell_1,\ldots,\ell_d}
(d\boldsymbol x)
=1,
\qquad j=1,\ldots,d,
\]
and 
\[
M_2(\ell_1,\ldots,\ell_d)
:=
\max_{1\leq j\leq d}
\int_{\mathbb R_+^d\setminus\{\boldsymbol0\}}
x_j^2\,
\widehat\nu_{\ell_1,\ldots,\ell_d}
(d\boldsymbol x)
<\infty.
\]
Then we consider hereafter the GSRF defined in Def \ref{defgen}, with power spectrum $\left\{ C_{\ell
}\right\} $, and with the following characteristic function of spherical harmonic coefficients, for any choice of multipoles $\ell_1,...,\ell_d$, $d \in \mathbb N$, 
\begin{equation}
\Phi^{(\gamma)}_{\boldsymbol{a}_{\ell _{1\cdot}},...,\boldsymbol{a}_{\ell
_{d\cdot}}}(\boldsymbol{\xi}_{\ell_1\cdot},\boldsymbol{\xi}_{\ell_2\cdot},...,\boldsymbol{\xi}_{\ell_d\cdot})=E_\beta\left(-\varphi^{(\gamma)}_{\ell_1,...,\ell_d}(C_{\ell_1}
\Vert\xi_{\ell_1\cdot}\Vert^2/2,...,C_{\ell_d}
\Vert\xi_{\ell_d\cdot}\Vert^2/2)\right),\label{def:gamma}
\end{equation}%
for $\beta \in (0,1]$, $\xi_i \in\mathbb{R}$, $i=1,...,d$.
\end{condition}

\begin{remark} Condition \ref{Luglio}
can be viewed as the introduction of a perturbation into a baseline Gaussian law. Indeed, let us focus on the special case where $\beta=1$ and $d=1$; under these circumstances, we obtain the characteristic function given in Def. \ref{defgen}, in the form \eqref{def:gamma},
becomes, for any given $\ell$,
\begin{eqnarray}
\Phi^{(\gamma)} _{\boldsymbol a_{\ell }}(\boldsymbol{\xi}_\ell)&=&\exp\left\{-\varphi^{(\gamma)}_\ell(C_{\ell}
\Vert\boldsymbol{\xi}_\ell\Vert^2/2)\right\}\notag \\
&=&\exp\left\{-(1-\gamma)\Vert\boldsymbol{\xi}_\ell\Vert^2/2  \right\}\exp{\left\{-\gamma\int_{\mathbb R_+}(1-e^{-x\Vert\boldsymbol{\xi}_\ell\Vert^2/2})\widehat \nu(dx)\right\}}\notag \\
&=& \Phi_{\boldsymbol Z_{1-\gamma}}(\boldsymbol \xi_\ell)\Phi_{\boldsymbol B_\gamma}(\boldsymbol\xi_\ell).\notag
\end{eqnarray}%
Therefore, we are dealing with the sum of two independent r.v.'s, where \(\boldsymbol  Z_{1-\gamma} \) is a zero-mean, \(2\ell+1\)-dimensional Gaussian vector, with covariance matrix \((1-\gamma)I_{2\ell+1}\) and \(\boldsymbol  B_{\gamma}\) is a vector of independent Brownian motions, evaluated on a common time which follows a compound Poisson distribution with jumps' law \(\gamma\widehat\nu(\cdot)\). Similar arguments can be developed in the case $d\neq1$.
\end{remark}

We are now ready to establish our main results, which provide an explicit quantitative bound between the characteristic functions of GSRFs and their Gaussian counterpart, in terms of the non-Gaussianity parameters $\beta-1$ and $\gamma$. In particular, we shall focus on the characteristic function distance, which is given by
\begin{equation}
D_{CF}[\boldsymbol{X},\boldsymbol{Y}] :=\sup_{\boldsymbol{\xi}\in\mathbb{R}^d}|\Phi_{\boldsymbol{X}}(\boldsymbol{\xi}) -\Phi_{\boldsymbol{Y}}(\boldsymbol{\xi})| ,   
\end{equation}
where $\boldsymbol{X}$ and $\boldsymbol{Y}$ are $d$-dimensional random vectors.

\begin{theorem}\label{corr}
Let the Laplace exponent \(\varphi_{\ell_1,\ldots,\ell_d}^{(\gamma)}
(\boldsymbol z)\) satisfy  Condition \ref{Luglio} and let \((\boldsymbol{a}_{\ell _{1\cdot}},...,\boldsymbol{a}_{\ell
_{d\cdot}})\) and \((\boldsymbol{Z}_{\ell _{1\cdot}},...,\boldsymbol{Z}_{\ell
_{d\cdot}})\) denote the spherical harmonic coefficients of the GSRF defined in Def. \ref{defgen}, with characteristic function \eqref{def:gamma}, and a Gaussian standard vector of the same dimension, respectively. Then the following upper bound holds, as \(\beta\uparrow1\),

\begin{equation}\label{distchf}
    D_{CF}[(\boldsymbol{a}_{\ell _{1\cdot}},...,\boldsymbol{a}_{\ell
_{d\cdot}}),(\boldsymbol{Z}_{\ell _{1\cdot}},...,\boldsymbol{Z}_{\ell
_{d\cdot}})]\leq
(1-\beta) C_{\ell_1,\ldots,\ell_d}^{(\gamma)}
+
\frac{
8\gamma
}{e^2}
M_2(\ell_1,\ldots,\ell_d)
+
o(1-\beta),
\end{equation}
where, for \(F(x)
:=
1-e^{-x}
-xe^{-x}
\bigl(
\operatorname{Ei}(x)-\log x
\bigr),\) \(x\in \mathbb R_+,\)
\[
C_{\ell_1,\ldots,\ell_d}^{(\gamma)}
:=
\sup_{\boldsymbol z\in\mathbb R_+^d}
\left|
F\left(
\varphi_{\ell_1,\ldots,\ell_d}^{(\gamma)}
(\boldsymbol z)
\right)
\right|=
\begin{cases}
C_*,
&
\displaystyle
\textit{if }\sup_{\boldsymbol z\in\mathbb R_+^d}
\varphi_{\ell_1,\ldots,\ell_d}^{(\gamma)}
(\boldsymbol z)
=+\infty,
\\
\displaystyle
\sup_{0<u<L_{\ell_1,\ldots,\ell_d}^{(\gamma)}}
|F(u)|< C_*,
&
\displaystyle
\textit{if }L_{\ell_1,\ldots,\ell_d}^{(\gamma)}
:=
\sup_{\boldsymbol z\in\mathbb R_+^d}
\varphi_{\ell_1,\ldots,\ell_d}^{(\gamma)}
(\boldsymbol z)
<\infty.
\end{cases}
\]
\end{theorem}

\begin{proof}
We write, for every
\(\boldsymbol z\in
\mathbb R_+^d\setminus\{\boldsymbol0\}\),
\[
r:=\|\boldsymbol z\|_1,
\qquad
\boldsymbol\theta
:=
\frac{\boldsymbol z}{\|\boldsymbol z\|_1},
\]
so that \(\boldsymbol z=r\boldsymbol\theta\), \(r>0\), and
\(
\boldsymbol\theta\in
\Delta_{d-1}
:=
\left\{
\boldsymbol\theta\in\mathbb R_+^d:
\|\boldsymbol\theta\|_1=1
\right\}.
\)
Moreover, for every fixed \(r>0\) and 
\(\boldsymbol\theta\in\Delta_{d-1}\), we define
\(
\varphi_{\boldsymbol\theta}^{(\gamma)}(r)
:=
\varphi_{\ell_1,\ldots,\ell_d}^{(\gamma)}
(r\boldsymbol\theta),
\)
and we write, for brevity,
\(\widehat\nu
:=
\widehat\nu_{\ell_1,\ldots,\ell_d}.
\)
Then
\[
\varphi_{\boldsymbol\theta}^{(\gamma)}(r)
=
(1-\gamma)r
+
\gamma
\int_{\mathbb R_+^d\setminus\{\boldsymbol0\}}
\left(
1-e^{-r\langle\boldsymbol\theta,\boldsymbol x\rangle}
\right)
\widehat\nu(d\boldsymbol x).
\]
Since
\[
\int_{\mathbb R_+^d\setminus\{\boldsymbol0\}}
\langle\boldsymbol\theta,\boldsymbol x\rangle
\widehat\nu(d\boldsymbol x)
=
\sum_{j=1}^d\theta_j
\int_{\mathbb R_+^d\setminus\{\boldsymbol0\}}
x_j\,\widehat\nu(d\boldsymbol x)
=1,
\]
we can write
\[
\begin{split}
r-\varphi_{\boldsymbol\theta}^{(\gamma)}(r)
&=
\gamma
\int_{\mathbb R_+^d\setminus\{\boldsymbol0\}}
\left(
r\langle\boldsymbol\theta,\boldsymbol x\rangle
-1
+e^{-r\langle\boldsymbol\theta,\boldsymbol x\rangle}
\right)
\widehat\nu(d\boldsymbol x).
\end{split}
\]
By recalling that
\(
0\leq u-1+e^{-u}\leq u^2/2,
\) for \( u\geq0,
\)
we obtain
\begin{equation}\label{gamma-bound}
0
\leq
r-\varphi_{\boldsymbol\theta}^{(\gamma)}(r)
\leq
\frac{\gamma r^2}{2}
\int_{\mathbb R_+^d\setminus\{\boldsymbol0\}}
\langle\boldsymbol\theta,\boldsymbol x\rangle^2
\widehat\nu(d\boldsymbol x)\leq
\frac{\gamma r^2}{2}\sum_{j=1}^d\theta_j
\int_{\mathbb R_+^d\setminus\{\boldsymbol0\}} x_j^2\,\widehat\nu(d\boldsymbol x)
\leq
\frac{\gamma r^2}{2}M_2(\ell_1,\ldots,\ell_d),
\end{equation}
which follows by considering the convexity of \(x\mapsto x^2\) and the consequent inequality
\(
\langle\boldsymbol\theta,\boldsymbol x\rangle^2
\leq
\sum_{j=1}^d\theta_jx_j^2.
\)
We now write
\begin{equation}\label{sec}
\left|
E_\beta\left(
-\varphi_{\boldsymbol\theta}^{(\gamma)}(r)
\right)
-e^{-r}
\right|
\leq
\left|
E_\beta\left(
-\varphi_{\boldsymbol\theta}^{(\gamma)}(r)
\right)
-
e^{-\varphi_{\boldsymbol\theta}^{(\gamma)}(r)}
\right|
+
\left|
e^{-\varphi_{\boldsymbol\theta}^{(\gamma)}(r)}
-e^{-r}
\right|,
\end{equation}
so that, for the first term, Lemma \ref{lemsup} yields
\begin{equation}\label{try}
\sup_{\boldsymbol\theta\in\Delta_{d-1}}
\sup_{r>0}
\left|
E_\beta\left(
-\varphi_{\boldsymbol\theta}^{(\gamma)}(r)
\right)
-
e^{-\varphi_{\boldsymbol\theta}^{(\gamma)}(r)}
\right|
\leq
C_{\ell_1,\ldots,\ell_d}^{(\gamma)}
(1-\beta)
+
o(1-\beta),
\end{equation}
since
\[
\left\{
\varphi_{\boldsymbol\theta}^{(\gamma)}(r):
r>0,\
\boldsymbol\theta\in\Delta_{d-1}
\right\}
=
\left\{
\varphi_{\ell_1,\ldots,\ell_d}^{(\gamma)}
(\boldsymbol z):
\boldsymbol z\in
\mathbb R_+^d\setminus\{\boldsymbol0\}
\right\}.
\]
It remains to find an upper bound for
\[
D_{\ell_1,\ldots,\ell_d}^{(\gamma)}
:=
\sup_{\boldsymbol\theta\in\Delta_{d-1}}
\sup_{r>0}
\left|
e^{-\varphi_{\boldsymbol\theta}^{(\gamma)}(r)}
-e^{-r}
\right|.
\]
By \eqref{gamma-bound}, we have that
\(
\varphi_{\boldsymbol\theta}^{(\gamma)}(r)
\leq r,
\)
where the equality holds if 
\(
\mu
:=
\gamma
M_2(\ell_1,\ldots,\ell_d)=0,
\)
and the second term in \eqref{sec} vanishes identically. Thus we assume that
\(\mu>0\). For
\(
0<r\leq \mu^{-1},
\)
formula \eqref{gamma-bound} gives
\[
\varphi_{\boldsymbol\theta}^{(\gamma)}(r)
\geq
r-\frac{\mu}{2}r^2
\geq
\frac r2.
\]
Thus,
\[
0
\leq
e^{-\varphi_{\boldsymbol\theta}^{(\gamma)}(r)}
\left(
1-e^{-(r-\varphi_{\boldsymbol\theta}^{(\gamma)}(r))}
\right)
\leq
e^{-\varphi_{\boldsymbol\theta}^{(\gamma)}(r)}
\left(
r-\varphi_{\boldsymbol\theta}^{(\gamma)}(r)
\right)
\leq
\frac{
\mu
}{2}
r^2e^{-r/2}.
\]
By recalling 
that \(
\sup_{z>0}z^2e^{-z/2}=16e^{-2}
\),
we also have
\[
\sup_{\boldsymbol\theta\in\Delta_{d-1}}
\sup_{0<r\leq \mu^{-1}}
\left(
e^{-\varphi_{\boldsymbol\theta}^{(\gamma)}(r)}
-e^{-r}
\right)
\leq
\frac8{e^2}
\mu.
\]
In the other case, for
\(
r\geq \mu^{-1},
\)
the monotonicity of
\(\varphi_{\boldsymbol\theta}^{(\gamma)}(\cdot)\) gives
\(
e^{-\varphi_{\boldsymbol\theta}^{(\gamma)}(r)}
\leq
e^{
-\varphi_{\boldsymbol\theta}^{(\gamma)}
(\mu^{-1})
}.
\)
Moreover,
\[
\varphi_{\boldsymbol\theta}^{(\gamma)}
(\mu^{-1})
\geq
\frac1{
2\mu
},
\]
and hence
\[
0
\leq
e^{-\varphi_{\boldsymbol\theta}^{(\gamma)}(r)}
-e^{-r}
\leq
e^{-1/(2\mu)}.
\]
If
\(
0<\mu\leq1,
\)
then
\[
e^{-1/(2\mu)}
\leq
\frac2e
\mu
\leq
\frac8{e^2}
\mu.
\]
If
\(
\mu>1,
\)
we use the trivial upper bound
\[
e^{-\varphi_{\boldsymbol\theta}^{(\gamma)}(r)}
-e^{-r}
\leq1
\leq
\frac8{e^2}
\mu.
\]
Therefore,
\[
D_{\ell_1,\ldots,\ell_d}^{(\gamma)}
\leq
\frac{
8\gamma
}{e^2}
M_2(\ell_1,\ldots,\ell_d),
\]
which, combined with \eqref{try} yields 
\begin{equation}
\label{ubi}
\sup_{\boldsymbol z\in
\mathbb R_+^d\setminus\{\boldsymbol0\}}
\left|
E_\beta\left(
-\varphi_{\ell_1,\ldots,\ell_d}^{(\gamma)}
(\boldsymbol z)
\right)
-e^{-\|\boldsymbol z\|_1}
\right|
\leq
C_{\ell_1,\ldots,\ell_d}^{(\gamma)}
(1-\beta)
+
\frac{
8\gamma
}{e^2}
M_2(\ell_1,\ldots,\ell_d)
+
o(1-\beta).
\end{equation}
Equation
\eqref{distchf} follows from the definition of the distance $D_{CF}$.
\end{proof}

\begin{remark} In the special case \(\beta=1\),
the parametrization
\[
\varphi_{\ell_1,\ldots,\ell_d}^{(\gamma)}
(\boldsymbol z)
=
(1-\gamma)
\|\boldsymbol z\|_1
+
\gamma
\widehat\varphi_{\ell_1,\ldots,\ell_d}
(\boldsymbol z),
\qquad
0\leq\gamma\leq1,
\]
provides an interpolation between the Gaussian case and the
multivariate subordinators considered in the previous sections.
Indeed, for
\(\gamma=0\), one has
\(
\varphi_{\ell_1,\ldots,\ell_d}^{(0)}
(\boldsymbol z)
=
\|\boldsymbol z\|_1,
\)
which is the Laplace exponent of the deterministic process
\(
t\longmapsto t(1,\ldots,1),
\)
and hence corresponds to the Gaussian
case. On the other hand, for
\(\gamma=1\),
\(
\varphi_{\ell_1,\ldots,\ell_d}^{(1)}
(\boldsymbol z)
=
\widehat\varphi_{\ell_1,\ldots,\ell_d}
(\boldsymbol z),
\)
so that the original multivariate subordinator is recovered.

More generally, if
\(\widehat{\boldsymbol{\mathcal H}}_{\ell_1,\ldots,\ell_d}\)
is a multivariate subordinator with Laplace exponent
\(\widehat\varphi_{\ell_1,\ldots,\ell_d}\), then
\[
\boldsymbol{\mathcal H}_{\ell_1,\ldots,\ell_d}^{(\gamma)}(t)
\overset{\mathrm{f.d.d.}}{=}
(1-\gamma)
t(1,\ldots,1)
+
\widehat{\boldsymbol{\mathcal H}}_{\ell_1,\ldots,\ell_d}
\bigl(
\gamma t
\bigr).
\]
Thus, for \(0<\gamma<1\), the process contains
both a deterministic drift component and the jump structure of the
original multivariate subordinator.

In particular, by choosing
\(\widehat\varphi_{\ell_1,\ldots,\ell_d}\) as the normalized Laplace
exponent of the corresponding process, this construction includes the
multivariate Gamma, inverse Gaussian, compound Poisson and tempered
stable subordinators considered in Appendix A. These models satisfy the
assumptions of the previous lemma whenever their first moments are normalized
and the second moments of the corresponding L\'evy measure are finite.
The multivariate stable subordinator is also included in the
parametrized family for
\(\gamma=1\), but it does not satisfy the
finite-second-moment assumption and must therefore be treated
separately.

For the family of multivariate TSS's introduced in Section \ref{secdepsub}, the normalizing condition on first moment is satisfied only under symmetry assumption on the spectral measure \(\lambda(\cdot)\) and of the function $b(\cdot)$ in \eqref{tss}, i.e. the integral
\(
I_{\ell_1,...,\ell_d}:=\int_{\mathbb{S}_+^{d-1}}b_{\ell_1,...,\ell_d}(\boldsymbol{v})^{\alpha-1}\boldsymbol{v}_j\lambda_{\ell_1,...,\ell_d}(d\boldsymbol{v}) 
\) must be independent of $j$. In this case,  Lemma
\ref{corr} holds with
\[
M_2(\ell_1,\ldots,\ell_d)
=\frac{1-\alpha}{I_{\ell_1,...,\ell_d}}
\max_{1\leq j\leq d}
\int_{\mathbb S^{d-1}}
b_{\ell_1,...,\ell_d}(\boldsymbol v)^{\alpha-2}\,
v_j^2
\lambda_{\ell_1,...,\ell_d}(d\boldsymbol v)
.
\]
\end{remark}

When the moments of the multivariate subordinators are not bounded and in the special case of a stable $d$-dimensional subordinator, the assumption in Condition \ref{Luglio} and the bound given in Lemma \ref{corr} must be replaced by the following ones.

\begin{remark}
Clearly, when the spherical harmonic coefficients have infinite variance, the associated random field cannot have finite variance either, and hence the expansion \eqref{eq:SpecRap} does not hold in general. However, for a \textit{finite} set of multipoles $\ell=0,1,2,...,L<\infty,$ equation \eqref{eq:SpecRap} can be taken as a definition, the equality holding in the a.s. sense. In these cases, the field is still isotropic in the strong-sense, although not wide-sense (see \cite{MP11}, ch. 5, p. 115).
\end{remark}

\begin{lemma}\label{Sept}
Let $0<\beta<1$, and let, for any $\ell_1,\ldots,\ell_d$, \((\boldsymbol{Z}_{\ell _{1\cdot}},...,\boldsymbol{Z}_{\ell
_{d\cdot}})\) be a Gaussian standard vector of dimension $d$. Let moreover 
\((\boldsymbol{a}_{\ell _{1\cdot}},...,\boldsymbol{a}_{\ell
_{d\cdot}})\)
denote the spherical harmonic coefficients of the GSRF defined in Def. \ref{defgen}, where
\(
\boldsymbol {\mathcal H}_{\ell_1,\ldots,\ell_d}(t)
=
\bigl(\mathcal H_{\ell_1}(t),\ldots, \mathcal H_{\ell_d}(t)\bigr)
\), for \(t\geq0,
\) is a non-degenerate, strictly $\alpha$-stable $d$-dimensional subordinator,
with $0<\alpha<1$ and Laplace exponent
\(
\varphi_{\ell_1,\ldots,\ell_d}:
\mathbb R_+^d\to\mathbb R_+.
\)

Then, as $\beta\uparrow1$, the following upper bound holds
\begin{equation}
  D_{CF}[(\boldsymbol{a}_{\ell _{1\cdot}},...,\boldsymbol{a}_{\ell
_{d\cdot}}),(\boldsymbol{Z}_{\ell _{1\cdot}},...,\boldsymbol{Z}_{\ell
_{d\cdot}})]\leq
C^*(1-\beta)
+
\frac{1-\alpha}{\alpha}
\mathcal M_{\ell_1,\ldots,\ell_d}
+
\mathcal Q_{\ell_1,\ldots,\ell_d}
+
o(1-\beta),\label{res}
\end{equation}
where
\[
\mathcal M_{\ell_1,\ldots,\ell_d}
:=
\sup_{\boldsymbol\theta\in\Delta_{d-1}}
\sup_{u>0}
u e^{-u}
\left|
\log\left(
\frac{u}{
\varphi_{\ell_1,\ldots,\ell_d}(\boldsymbol\theta)}
\right)
\right|
\]
and
\[
\mathcal Q_{\ell_1,\ldots,\ell_d}
:=
\sup_{\boldsymbol\theta\in\Delta_{d-1}}
\frac{
\left|
\varphi_{\ell_1,\ldots,\ell_d}(\boldsymbol\theta)-1
\right|
}{
e\min\{
\varphi_{\ell_1,\ldots,\ell_d}(\boldsymbol\theta),1
\}
}.
\]
\end{lemma}
\begin{proof}
Since, by assumption, $\boldsymbol {\mathcal H}_{\ell_1,\ldots,\ell_d}$ is non-degenerate, i.e. then \[
\inf_{\boldsymbol\theta\in\Delta_{d-1}}
\varphi_{\ell_1,\ldots,\ell_d}(\boldsymbol\theta)>0.
\]
Moreover, it is strictly
$\alpha$-stable, so that its Laplace exponent is homogeneous of degree
$\alpha$, namely
\[
\varphi_{\ell_1,\ldots,\ell_d}(\boldsymbol z)
=
r^\alpha
\varphi_{\ell_1,\ldots,\ell_d}(\boldsymbol\theta),\qquad r:=\Vert\boldsymbol{z}\Vert>0,\boldsymbol\theta
:=
\boldsymbol z/\|\boldsymbol z\|_1 .
\]
Hence
we can write, for $\boldsymbol{z}=r\boldsymbol{\theta}$, \begin{equation}\label{6dieci}
\begin{aligned}
&
\left|
E_\beta\!\left(
-\varphi_{\ell_1,\ldots,\ell_d}(\boldsymbol z)
\right)
-e^{-\|\boldsymbol z\|_1}
\right|
\\
&\leq
\left|
E_\beta\!\left(
-r^\alpha\varphi_{\ell_1,\ldots,\ell_d}(\boldsymbol\theta)
\right)
-
e^{-r^\alpha\varphi_{\ell_1,\ldots,\ell_d}(\boldsymbol\theta)}
\right|
+
\left|
e^{-r^\alpha\varphi_{\ell_1,\ldots,\ell_d}(\boldsymbol\theta)}
-
e^{-r\varphi_{\ell_1,\ldots,\ell_d}(\boldsymbol\theta)}
\right|\notag
\\
&\quad+
\left|
e^{-r\varphi_{\ell_1,\ldots,\ell_d}(\boldsymbol\theta)}
-e^{-r}
\right|.
\end{aligned}
\end{equation}
For the first term, the change of variable
\(
u=r^\alpha
\varphi_{\ell_1,\ldots,\ell_d}(\boldsymbol\theta)
\), together with Lemma \ref{lemsup}, yield
\begin{equation}
\sup_{r>0}
\left|
E_\beta\!\left(
-r^\alpha\varphi_{\ell_1,\ldots,\ell_d}(\boldsymbol\theta)
\right)
-
e^{-r^\alpha\varphi_{\ell_1,\ldots,\ell_d}(\boldsymbol\theta)}
\right|
=
C^*(1-\beta)+o(1-\beta).
\label{6undici}
\end{equation}
For the second term, we apply the mean value theorem in order to write, for some $\xi\in(\alpha,1)$,
\[
\left|
e^{-r^\alpha\varphi_{\ell_1,\ldots,\ell_d}(\boldsymbol\theta)}
-
e^{-r\varphi_{\ell_1,\ldots,\ell_d}(\boldsymbol\theta)}
\right|
\leq
(1-\alpha)
\varphi_{\ell_1,\ldots,\ell_d}(\boldsymbol\theta)
r^\xi|\log r|\,
e^{-\varphi_{\ell_1,\ldots,\ell_d}(\boldsymbol\theta)r^\xi},
\]
for any fixed $r>0$. Then, for
\(
v=
\varphi_{\ell_1,\ldots,\ell_d}(\boldsymbol\theta)r^\xi,
\)
we have
\[
|\log r|
=
\frac1{\xi}
\left|
\log
\frac{v}{
\varphi_{\ell_1,\ldots,\ell_d}(\boldsymbol\theta)}
\right|.
\]
Since $\xi\geq\alpha$,
\begin{equation}
\sup_{r>0}
\left|
e^{-r^\alpha\varphi_{\ell_1,\ldots,\ell_d}(\boldsymbol\theta)}
-
e^{-r\varphi_{\ell_1,\ldots,\ell_d}(\boldsymbol\theta)}
\right|
\leq
\frac{1-\alpha}{\alpha}
\sup_{v>0}
v e^{-v}
\left|
\log
\frac{v}{
\varphi_{\ell_1,\ldots,\ell_d}(\boldsymbol\theta)}
\right|.
\label{6dodici}
\end{equation}
Finally, by applying again the mean value theorem to the third term, we get
\[
\left|
e^{-r\varphi_{\ell_1,\ldots,\ell_d}(\boldsymbol\theta)}
-e^{-r}
\right|\leq
\left|
\varphi_{\ell_1,\ldots,\ell_d}(\boldsymbol\theta)-1
\right|
r
e^{-\min\{
\varphi_{\ell_1,\ldots,\ell_d}(\boldsymbol\theta),1
\}r}.
\]
Since
\[
\sup_{r>0}re^{-ar}=\frac1{ea},
\qquad a>0,
\]
we obtain
\begin{equation}
\sup_{r>0}
\left|
e^{-r\varphi_{\ell_1,\ldots,\ell_d}(\boldsymbol\theta)}
-e^{-r}
\right|
\leq
\frac{
\left|
\varphi_{\ell_1,\ldots,\ell_d}(\boldsymbol\theta)-1
\right|
}{
e\min\{
\varphi_{\ell_1,\ldots,\ell_d}(\boldsymbol\theta),1
\}
}.
\label{6tredici}
\end{equation}
Taking the supremum over
$\boldsymbol\theta\in\Delta_{d-1}$ in
\eqref{6dieci}, \eqref{6dodici} and \eqref{6tredici} gives the upper bound in \eqref{res} for the distance
\begin{equation}\notag
    D_{CF}[(\boldsymbol{a}_{\ell _{1\cdot}},...,\boldsymbol{a}_{\ell
_{d\cdot}}),(\boldsymbol{Z}_{\ell _{1\cdot}},...,\boldsymbol{Z}_{\ell
_{d\cdot}})]=\sup_{\boldsymbol z\in\mathbb R_+^d\setminus\{\boldsymbol0\}}\left|E_\beta\!\left(-\varphi_{\ell_1,\ldots,\ell_d}(\boldsymbol z)\right)-e^{-\|\boldsymbol z\|_1}\right|.
\end{equation}

\end{proof}

\begin{remark}In some special cases, the previous result can be considerably simplified. 
A particularly simple example is obtained when the components of the
multivariate $\alpha$-stable subordinator are completely dependent. More precisely, let
$\{\mathcal H_\alpha(t),\,t\geq0\}$ be a one-dimensional strictly
$\alpha$-stable subordinator, $0<\alpha<1$, with Laplace transform
\[
\mathbb E\left[e^{-z\mathcal H_\alpha(t)}\right]
=
e^{-tc z^\alpha},
\qquad z\geq0,
\]
for some $c>0$, and define
\(
\boldsymbol {\mathcal H}(t)
=
\bigl(\mathcal H_\alpha(t),\ldots,\mathcal H_\alpha(t)\bigr),
\) for any $t \geq 0.$
Thus, all the components are driven by the same stable subordinator and
are therefore completely dependent. In this case,
\[
\mathbb E\left[
e^{-\langle\boldsymbol z,\boldsymbol {\mathcal{H}}(t)\rangle}
\right]
=
\mathbb E\left[
e^{-(z_1+\cdots+z_d)\mathcal H_\alpha(t)}
\right]
=
\exp\left\{
-tc(z_1+\cdots+z_d)^\alpha
\right\},
\]
so that the multivariate Laplace exponent is
\(
\varphi(\boldsymbol z)
=
c\left(z_1+\cdots+z_d\right)^\alpha.
\)

For $\boldsymbol\theta\in\Delta_{d-1}$, since
$\theta_1+\cdots+\theta_d=1$, we have that
\(
\varphi(\boldsymbol\theta)=c.
\)
Consequently, the quantity appearing in Lemma \ref{Sept} reduces, under the standard normalization $c=1$, to
\[
\mathcal M
=
\sup_{u>0}u e^{-u}|\log u|
.
\]
Hence, in the completely dependent normalized case,
$\mathcal M$ is a universal constant, independent of both the dimension
$d$ and the stability index $\alpha$.
 By numerical methods the previous constant can be evaluated as $\mathcal M\simeq 0.270053$. On the other hand, in this special case it is
 \[\mathcal Q_{\ell_1,\ldots,\ell_d}=\frac{|c-1|}{e\, \min\{c,1\}},\]
 so that, for $c=1$ it vanishes. Therefore, the bound in Lemma \ref{Sept} simplifies into
 \begin{equation}
\sup_{\boldsymbol z\in
\mathbb R_+^d\setminus\{\boldsymbol0\}}
\left|
E_\beta\!\left(
-\varphi_{\ell_1,\ldots,\ell_d}(\boldsymbol z)
\right)
-e^{-\|\boldsymbol z\|_1}
\right|
\leq
C^*(1-\beta)
+
\frac{0.27(1-\alpha)}{\alpha}
+
o(1-\beta).\notag
\end{equation}
\end{remark}

\section{Generalizations of the random variance case }\label{sec:GGSRF}
Whole our paper has been concerned with different form of GSRFs related to Mittag-Leffler characteristic functions; it is indeed possible to consider many other models which are made non-Gaussian by allowing for a random global variance. 
More precisely, we can extend the previous results by considering more general completely monotone functions, rather than the (one-parameter) Mittag-Leffler. Indeed, it is well-known that, by the Bernstein theorem, if $G:\mathbb 
 R_+ \to \mathbb{R}_+$ is completely monotone, then there exists a unique measure $\mu $ such that%
\begin{equation*}
G(x)=\int_{0}^{+\infty }e^{-xt}\mu_G (dt), \qquad x>0,
\end{equation*}%
for a unique measure $\mu_G$ and thus it coincides with its Laplace transform (see \cite{SCH}, p.3). Moreover, it follows from the Bochner-Minlos theorem that there exists
a unique functional $\Phi $ such that $\Phi (\boldsymbol \xi ):=G(\left\Vert
\boldsymbol \xi \right\Vert ^{2})\in
L^{2}(\mathbb{R}^d)$ 
and
\begin{equation*}
\Phi (\xi )=\int_{\mathbb{R}^d}e^{i\left\langle \boldsymbol x,\boldsymbol \xi \right\rangle }d\mu_G
(\boldsymbol x)\text{,\qquad }\boldsymbol \xi \in \mathbb{R}^d.
\end{equation*}
Let us now define the probability measure $\overline{\mu}_G$ by the following normalized  completely monotone function
\[\overline{G}(x):=\frac{G(x)}{G(0)}=\int_{0}^\infty e^{-xt}\overline{\mu}_G(dt), \qquad x \geq 0.\]

We start by generalizing the Definition \ref{defGRF}, as follows.

\begin{definition}\label{defGRFbis}
Let $\overline G:\mathbb 
 R_+ \to \mathbb{R}_+$ be a normalized completely monotone function.     
A spherical random field $\mathcal{T}(\boldsymbol x),$ $\boldsymbol{x}\in \mathbb S^2$, with power spectrum $\left\{ C_{\ell
}\right\} $, is a \textit{Generalized Grey} random filed (G-GSRF), with random global variance, if, for any choice of multipoles $\ell_1,...,\ell_d$, we have that%
\begin{equation}\label{defGRFeqbis}
\Phi _{\boldsymbol{a}_{\ell _{1\cdot}},...,\boldsymbol{a}_{\ell
_{d\cdot}}}(\boldsymbol{\xi}_{\ell_1\cdot},\boldsymbol{\xi}_{\ell_2\cdot},...,\boldsymbol{\xi}_{\ell_d\cdot})=\overline G\left(\frac{1}{2K}\sum_{i=1}^dC_{\ell_i}
\Vert\boldsymbol{\xi}_{\ell_i\cdot}\Vert^2\right),
\end{equation}%
where the normalizing constant $K$ coincides with the mean of the measure $\mu_G$, i.e. $K=-\overline G'(0)> 0$, {and $\boldsymbol \xi_i \in\mathbb{R}^{2\ell_i+1}$, $i=1,...,d$.}
\end{definition}

Thus, if we denote by $Z(\boldsymbol x),$ $\boldsymbol x \in \mathbb{S}^2$ the isotropic spherical Gaussian random field with mean $0$ and covariance function 
\[\rho_{i,j}:=Cov(Z(x_i),Z(x_j))=\sum_{lm}C_lY_{lm}(x_i)Y_{lm}(x_j)=\sum_l \frac{2l+1}{4\pi}C_lP_l(\langle x_i,x_j \rangle),\]
where $Z(\boldsymbol x)=\sum_{lm}a_{lm}Y_{lm}(\boldsymbol x),$ and $a_{lm} \sim N(0,C_l)$, the G-grey random field can be represented as 
\begin{equation}\mathcal{T}(\boldsymbol x)=(\mathbb EU)^{-1/2}\sqrt{U}Z(\boldsymbol x), \qquad \boldsymbol x \in \mathbb{S}^2,\label{deft}\end{equation} where $U$ is the r.v. with measure $\mu_G(\cdot)$ and, by assumption, $\mathbb EU\neq 0$ and finite.
Indeed, the joint characteristic function of the random vector $(\mathcal{T}(x_1),...,\mathcal{T}(x_d))$ is given by

\begin{eqnarray}
    \mathbb Ee^{i(\xi_1\mathcal{T}(x_1)+...+\xi_d \mathcal{T}(x_d))}&=&\int_0^\infty \mathbb Ee^{i\frac{u^{1/2}}{\sqrt{\mathbb E U}}(\xi_1\mathcal{T}(x_1)+...+\xi_d \mathcal{T}(x_d))}d\mu_G(u)\label{fi} \\
    &=& \int_0^\infty e^{-\frac{u}{2\mathbb E U}\sum_{i,j=1}^d\xi_i\xi_j\rho_{i,j}}d\mu_G(u)=\overline G\left(\frac{1}{2\mathbb E U}\sum_{i,j=1}^d\xi_i\xi_j\rho_{i,j}\right) \notag
\end{eqnarray}
which coincides with \eqref{defGRFeqbis} with $K=\mathbb{E}U$. 

We note that the tail behaviour of the harmonic coefficients is entirely encoded in the behaviour of the completely monotone function \(G\) near the origin.

Indeed, the function $
\overline G(s)=\mathbb E(e^{-sU})
$
is the Laplace transform of the mixing random variable \(U\); moreover, the spherical harmonic coefficients admit the representation
\[
a_{\ell m}
=
\sqrt{\frac{U}{\mathbb EU}}\, Z_{\ell m},
\]
where \(Z_{\ell m}\sim N(0,C_\ell)\).
Then, if we assume that
\[
1-\overline G(s)\sim c\, s^\alpha L(1/s),
\qquad s\downarrow 0,
\]
where \(0<\alpha<1\) and \(L(\cdot)\) is a slowly varying function at infinity,
the Tauberian theorem for Laplace transforms gives the following asymptotics for the tail of the distribution of $U$: 
\[
\mathbb P(U>u)
\sim
\frac{c}{\Gamma(1-\alpha)}
u^{-\alpha}L(u),
\qquad u\to\infty .
\]
It follows from standard results on
Gaussian scale mixtures that
\[
\mathbb P(|a_{\ell m}|>x)
\sim
\frac{c}{\Gamma(1-\alpha)}
\left(\frac{C_\ell}{\mathbb EU}\right)^\alpha
\mathbb E|Z|^{2\alpha}\,
x^{-2\alpha}L(x^2),
\qquad x\to\infty ,
\]
where \(Z\sim N(0,1)\).

Hence, the asymptotic behaviour of the completely monotone function
\(G\) at the origin completely determines the tail decay of the harmonic
coefficients. In particular, polynomial tails for the mixing distribution
yield polynomial tails for the spherical harmonic coefficients, while
light-tailed choices of \(G\) produce nearly Gaussian behaviour.

For any \(q\in\mathbb N\), the even-order moments of the spherical harmonic
coefficients satisfy
\[
\mathbb E[a_{\ell m}^{2q}]
=
C_\ell^q(2q-1)!!
\frac{\mathbb E(U^q)}{(\mathbb EU)^q}=
C_\ell^q(2q-1)!!
\frac{(-1)^q\overline G^{(q)}(0)}
{\left(-\overline G'(0)\right)^q}.
\]

In particular, the fourth-order moment is $$\frac{\mathbb E[a_{\ell m}^4]} {{3(\mathbb E[a_{\ell m}^2)^2}-1} = \frac{\overline G''(0)} {\left(\overline G'(0)\right)^2}-1.$$

Therefore, the completely monotone function \(G\) governs the distributional
features of the field, such as tails, moments, cumulants, and distance from the Gaussian behavior, whereas the
angular power spectrum \(\{C_\ell\}\) governs its geometric and sample-path
regularity. Indeed, as it is evident from the definition \eqref{deft}, the random field $\{\mathcal{T}(\boldsymbol{x})\}$ has H\"older continuous sample paths a.s. if and only if $\{Z(\boldsymbol{x})\}$ is H\"older continuous, which holds by assumption.


Another natural question concerns the behaviour of the field at high frequencies,
that is,  \(\ell\to\infty\). In the random global variance case we have that, for any $\ell$,
\[
a_{\ell m}
=
\sqrt{\frac{U}{\mathbb E U}}\,Z_{\ell m},
\qquad
Z_{\ell m}\sim N(0,C_\ell).
\]
Therefore, after normalization by the angular power spectrum,
\[
\frac{a_{\ell m}}{\sqrt{C_\ell}}
=
\sqrt{\frac{U}{\mathbb E U}}\,N_{\ell m},
\]
where \(N_{\ell m}\sim N(0,1)\).
Hence, 
the marginal law is independent of \(\ell\) and the non-Gaussianity
persists at all frequencies.

In Appendix B below, we review the choices of the function $\overline G$ in the definition of the G-GSRFs, which appear more relevant, discussing the corresponding fields and some of their most important properties.

\section{A space-time extension and the associated fractional evolution equation}
\label{sec:space-time}

The construction introduced in the previous sections admits a natural
space-time extension. Recall that the random variable $T_\beta$
entering the definition of the GSRF satisfies
\begin{equation}
T_\beta \overset{d}{=} \mathcal L_\beta(1),\label{tl}
\end{equation}
where $\{\mathcal L_\beta(t),t\geq0\}$ is the inverse of a $\beta$-stable
subordinator. This suggests replacing $T_\beta$ by the whole process
$\mathcal L_\beta(t)$ and viewing the GSRF introduced above as the 
marginal of a time-dependent random field.

In accordance with the real-valued harmonic representation adopted
throughout the paper, let
\[
\boldsymbol B_\ell(t)
=
\bigl(
B_{\ell,-\ell}(t),\ldots,B_{\ell,\ell}(t)
\bigr),
\qquad t\geq0,
\]
be independent $(2\ell+1)$-dimensional standard Brownian motions,
independent of the multivariate subordinator
$\boldsymbol {\mathcal{H}}=\{\mathcal H_\ell\}_{\ell\geq0}$ and of $\mathcal L_\beta$.
We define
\[
\boldsymbol a_{\ell,\cdot}(t)
:=
\sqrt{C_\ell}\,
\boldsymbol B_\ell
\bigl(\mathcal H_\ell(\mathcal L_\beta(t))\bigr),
\]
and introduce the space-time random field
\[
\mathcal T(x,t)
:=
\sum_{\ell\geq0}\sum_{m=-\ell}^{\ell}
a_{\ell m}(t)Y_{\ell m}(x),
\qquad
x\in\mathbb S^2,\quad t\geq0,
\]
whenever the series is well defined.

The variance-mixture representation given in Section \ref{sec:GreyInd} still holds, since, for any $t$,
conditionally on $\mathcal H_\ell(\mathcal L_\beta(t))$,
\[
\boldsymbol a_{\ell,\cdot}(t)
\overset{d}{=}
\sqrt{\mathcal H_\ell(\mathcal L_\beta(t))}\,
\boldsymbol Z_\ell\overset{d}{=}
\sqrt{\mathcal H_\ell(T_\beta)}\,\boldsymbol Z_\ell,
\]
where
\(
\boldsymbol Z_\ell
\sim N(0,C_\ell I_{2\ell+1}).
\)
Hence $\mathcal T(\cdot,1)$ has the same distribution as the GSRF
introduced in Section \ref{sec:GreyInd}.

We now derive the connection with
fractional evolution equations, where, as we will see, the relevant operator does not act on
the spatial variable $x\in\mathbb S^2$; rather, it acts on the
collection of harmonic coefficients. Its Fourier symbol is determined
by the joint Laplace exponent of the multivariate subordinator and
therefore preserves the dependence structure among different
multipoles. The inverse stable time change is responsible, in turn,
for the fractional derivative in time.

We then give the following definition.

\begin{definition}\label{defstGRF}
    
A spatio-temporal spherical random field with power spectrum $\left\{ C_{\ell
}\right\} $ is Grey if for any choice of multipoles $\ell_1,...,\ell_d$ we have that%
\begin{equation}\label{def:GeneralstGF}
\Phi _{\boldsymbol{a}_{\ell _{1\cdot}},...,\boldsymbol{a}_{\ell
_{d\cdot}}}(\boldsymbol{\xi}_{\ell_1\cdot},\boldsymbol{\xi}_{\ell_2\cdot},...,\boldsymbol{\xi}_{\ell_d\cdot},t)=E_\beta\left(-t^\beta\varphi_{\ell_1,...\ell_d}(C_{\ell_1}
\Vert\boldsymbol \xi_{\ell_1\cdot}\Vert^2/2,...,C_{\ell_d}
\Vert\boldsymbol \xi_{\ell_d\cdot}\Vert^2/2)\right),
\end{equation}%
for $\beta \in (0,1)$, $\xi_i \in\mathbb{R}$, $i=1,...,d$, and $\{\varphi_{\ell_1,...\ell_d}(.,...,.)\}_{\ell_1,...,\ell_d \in \mathbb{N}^d}$ a family of multivariate Laplace exponent functions.
\end{definition}

We now introduce the  operator $\mathcal A_{\boldsymbol\ell}$  defined by its symbol as
\[
\mathcal F
\left[
\mathcal A_{\boldsymbol\ell}f
\right]
(\boldsymbol\xi_{\ell_1},\ldots,
 \boldsymbol\xi_{\ell_d})
=
-
\varphi_{\ell_1,\ldots,\ell_d}
\left(
\frac{C_{\ell_1}}{2}
\|\boldsymbol\xi_{\ell_1}\|^2,
\ldots,
\frac{C_{\ell_d}}{2}
\|\boldsymbol\xi_{\ell_d}\|^2
\right)
\widehat f
(\boldsymbol\xi_{\ell_1},\ldots,
 \boldsymbol\xi_{\ell_d}),
\]
or, equivalently, in symbolic notation,
\begin{equation}
\mathcal A_{\boldsymbol\ell}
=
-
\varphi_{\ell_1,\ldots,\ell_d}
\left(
-\frac{C_{\ell_1}}{2}\Delta_{\boldsymbol a_{\ell_1}},
\ldots,
-\frac{C_{\ell_d}}{2}\Delta_{\boldsymbol a_{\ell_d}}
\right),\label{def:oper}
\end{equation}
where
\(
\Delta_{\boldsymbol a_\ell}
=
\sum_{m=-\ell}^{\ell}
\frac{\partial^2}{\partial a_{\ell m}^2}
\)
is the Euclidean Laplacian on the $(2\ell+1)$-dimensional coefficient
space associated with the $\ell$-th multipole.

Thus, we have the following result for the 
spatio-temporal GSRF.
\begin{proposition}
\label{prop:space-time-GSRF}
For fixed multipoles $\ell_1,\ldots,\ell_d$ the joint law $p_{\boldsymbol\ell}(\boldsymbol a,t)$ of the spherical harmonics of the spatio-temporal GSRF defined in \eqref{def:GeneralstGF} satisfies the Cauchy problem
\[
\begin{cases}
{}^C D_t^\beta
p_{\boldsymbol\ell}(\boldsymbol a,t)
=
\mathcal A_{\boldsymbol\ell}
p_{\boldsymbol\ell}(\boldsymbol a,t),
\\[1mm]
p_{\boldsymbol\ell}(\boldsymbol a,0)
=
\delta_{\boldsymbol 0}(\boldsymbol a),
\end{cases}
\]
where
\(
\boldsymbol a
=
(\boldsymbol a_{\ell_1},\ldots,
 \boldsymbol a_{\ell_d})
\in
\mathbb R^{D_{\boldsymbol\ell}},
\) for \(
D_{\boldsymbol\ell}
=
\sum_{j=1}^d(2\ell_j+1),
\) 
and $\mathcal A_{\boldsymbol{\ell}}$ is the operator defined in \eqref{def:oper}.
\end{proposition}

\begin{proof}
Conditionally on
$\boldsymbol {\mathcal H}_{\boldsymbol\ell}(\mathcal L_\beta(t))$, the vectors
$\boldsymbol a_{\ell_j}(t)$ are independent centered Gaussian vectors
with covariance matrices
\(
C_{\ell_j}H_{\ell_j}(L_\beta(t))
I_{2\ell_j+1}.
\)
Therefore,
\[
\mathbb E\left[
\exp\left\{
i\sum_{j=1}^d
\left\langle
\boldsymbol\xi_{\ell_j},
\boldsymbol a_{\ell_j}(t)
\right\rangle
\right\}
\,\middle|\,
\boldsymbol{\mathcal H}_{\boldsymbol\ell}(\mathcal L_\beta(t))
\right]
=
\exp\left\{
-\frac12
\sum_{j=1}^d
C_{\ell_j}
\|\boldsymbol\xi_{\ell_j}\|^2
\mathcal H_{\ell_j}(\mathcal L_\beta(t))
\right\},
\]
so that, conditioning on $\mathcal L_\beta(t)$, we obtain
\[
\mathbb E\left[
\exp\left\{
i\sum_{j=1}^d
\left\langle
\boldsymbol\xi_{\ell_j},
\boldsymbol a_{\ell_j}(t)
\right\rangle
\right\}
\,\middle|\,
\mathcal L_\beta(t)
\right]
 =
\exp\left\{
-\mathcal L_\beta(t)
\varphi_{\ell_1,\ldots,\ell_d}
\left(
\frac{C_{\ell_1}}{2}
\|\boldsymbol\xi_{\ell_1}\|^2,
\ldots,
\frac{C_{\ell_d}}{2}
\|\boldsymbol\xi_{\ell_d}\|^2
\right)
\right\}.
\]
By recalling that
\(
\mathbb E e^{-\lambda \mathcal L_\beta(t)}
=
E_\beta(-\lambda t^\beta)\), for \( \lambda\geq0,
\)
the previous expression coincides with the joint characteristic function given in \eqref{def:GeneralstGF}.
Using the eigenfunction identity given in \eqref{Caputo},
we obtain
\[
{}^C D_t^\beta
\widehat p_{\boldsymbol\ell}
(\boldsymbol\xi_{\ell_1},\ldots,
  \boldsymbol\xi_{\ell_d};t)
=
-
\varphi_{\ell_1,\ldots,\ell_d}
\left(
\frac{C_{\ell_1}}{2}
\|\boldsymbol\xi_{\ell_1}\|^2,
\ldots,
\frac{C_{\ell_d}}{2}
\|\boldsymbol\xi_{\ell_d}\|^2
\right)
\widehat p_{\boldsymbol\ell}
(\boldsymbol\xi_{\ell_1},\ldots,
 \boldsymbol\xi_{\ell_d};t),
\]
while the initial condition follows from $\mathcal L_\beta(0)=0$, a.s.

\end{proof}

\begin{remark}
\label{rem:stable-fractional-operator}
The form of the operator $\mathcal A_{\boldsymbol\ell}$ becomes
particularly transparent when the multivariate subordinator is
strictly $\alpha$-stable, with $0<\alpha<1$. In this case its joint
Laplace exponent is homogeneous of order $\alpha$, namely
\[
\varphi_{\ell_1,\ldots,\ell_d}
(r\kappa_1,\ldots,r\kappa_d)
=
r^\alpha
\varphi_{\ell_1,\ldots,\ell_d}
(\kappa_1,\ldots,\kappa_d),
\qquad r>0.
\]
Hence the operator
given in \eqref{def:oper}
can be interpreted as a fractional non-local operator on the space
of harmonic coefficients.

A particularly simple example is obtained when the components of the
stable subordinator are completely dependent and
\[
\varphi_{\ell_1,\ldots,\ell_d}
(\kappa_1,\ldots,\kappa_d)
=
(\kappa_1+\cdots+\kappa_d)^\alpha.
\]
In this case,
\[
\mathcal A_{\boldsymbol\ell}
=
-
\left(
-\frac{1}{2}
\sum_{j=1}^d
C_{\ell_j}\Delta_{\boldsymbol a_{\ell_j}}
\right)^\alpha.
\]
Thus, in this case, the space-time extension involves two distinct
fractional mechanisms: the parameter $\beta$ determines the order of
the Caputo time derivative, whereas $\alpha$ determines the order of
the fractional operator acting on the harmonic-coefficient space.

At the opposite extreme, if the components
$H_{\ell_1},\ldots,H_{\ell_d}$ are independent strictly
$\alpha$-stable subordinators with
\[
\varphi_{\ell_1,\ldots,\ell_d}
(\kappa_1,\ldots,\kappa_d)
=
\sum_{j=1}^d c_{\ell_j}\kappa_j^\alpha,
\qquad c_{\ell_j}>0,
\]
then the operator separates as
\[
\mathcal A_{\boldsymbol\ell}
=
-
\sum_{j=1}^d
c_{\ell_j}
\left(
-\frac{C_{\ell_j}}{2}
\Delta_{\boldsymbol a_{\ell_j}}
\right)^\alpha.
\]

More generally, for a dependent multivariate stable subordinator, the
joint Laplace exponent need not decompose into a sum of its marginal
Laplace exponents. 
\end{remark}

\begin{remark}
  Note that our approach here should not be confused with the construction of time-evolving fractional spherical random fields introduced in \cite{DOV} and further generalized in \cite{DOS}, where Gaussian fields are introduced as the solutions of fractional heat equations. In particular, in \cite{DOV}, Mittag-Leffler functions then emerge in the analytic expression of their angular power spectrum. This is entirely different from our purpose here, where we aim at isotropic Grey, i.e. non-Gaussian, fields.
\end{remark}

\medskip
The previous construction can be extended beyond the
Mittag-Leffler case by considering the (normalized) completely monotone function $\overline G(\cdot)$ introduced in Section \ref{sec:GGSRF}. This leads naturally to convolution-type
derivatives in the sense introduced in \cite{KOC}.
In particular, let us assume that
\begin{equation}
\overline G_f(\lambda)
:=\mathbb E[e^{-\lambda \Upsilon^f(1)}],
\qquad \lambda\geq 0,\label{def:new-G}
\end{equation}
where $\Upsilon^f$ is the inverse of a strictly increasing
subordinator, such that $\Upsilon^f(0)=0$ a.s., and with Laplace exponent
\[
f(s)=bs+\int_0^\infty(1-e^{-sr})\,\nu_f(dr),
\qquad f(0)=0.
\]

The assumption on $\overline G$ given in \eqref{def:new-G} is not
implied by complete monotonicity alone.
Therefore, the following result applies to the
subclass of generalized Grey spherical random
fields introduced in Section \ref{sec:GGSRF} whose mixing variable can
be represented as the inverse of a subordinator
at time one.

Let now $\overline{\nu}_f(t):=\nu_f((t,\infty))$ and introduce the
convolution-type derivative (\cite{KOC})
\begin{equation}
\mathbb D_t^f u(t)
=
b\,u'(t)+
\frac{d}{dt}\int_0^t
\overline \nu_f(t-s)\bigl(u(s)-u(0)\bigr)\,ds,\qquad t >0.
\label{eq:Kochubei_derivative}
\end{equation}
In the special case $f(s)=s^\beta$,
$0<\beta<1$, we get \(
\overline G_f(\lambda)=E_\beta(-\lambda).
\) and the operator $\mathbb D_t^f$
reduces to the Caputo derivative of order $\beta$.

Let $\{B_\ell\}_{\ell\geq 0}$ be independent
standard Brownian motions of dimensions $2\ell+1$,
independent of $\Upsilon^f$, and define
\[
a_{\ell,\cdot}^{\,f}(t)
:=
\sqrt{C_\ell}\,
B_\ell\left(\frac{\Upsilon^f(t)}{K}\right),
\]
then the joint characteristic
function of \begin{equation}\label{vectora}
\bigl(
a_{\ell_1,\cdot}^{\,f}(t),\ldots,
a_{\ell_d,\cdot}^{\,f}(t)
\bigr)
\end{equation}  coincides, for $t=1$, with that of the generalized
Grey spherical random field in Definition \ref{defGRFbis}, given in \eqref{defGRFeqbis}. In accordance with the normalization used there, 
we set
\[
K=-\overline G_f'(0)
=\mathbb E[\Upsilon^f(1)]\in(0,\infty).
\]
Let $\{B_\ell\}_{\ell\geq0}$ be independent standard
Brownian motions in $\mathbb R^{2\ell+1}$,
independent of $\Upsilon^f$, and define
\[
a_{\ell,\cdot}^f(t)
:=\sqrt{C_\ell}\,
B_\ell\left(\frac{\Upsilon^f(t)}{K}\right).
\]
Then, at $t=1$, the joint characteristic function
of $(a_{\ell_1,\cdot}^f(1),\ldots,
a_{\ell_d,\cdot}^f(1))$ is
\[
\overline G_f\left(
\frac{1}{2K}\sum_{j=1}^d
C_{\ell_j}\|\xi_{\ell_j}\|^2
\right),
\]
which coincides with (7.1).

\begin{proposition}
For fixed multipoles $\ell_1,\ldots,\ell_d$,
let $p_{\boldsymbol\ell}^{\,f}(\mathbf a,t)$
denote the joint law of \eqref{vectora}. 
Then $p_{\boldsymbol\ell}^{\,f}$ satisfies
the generalized Cauchy problem
\begin{equation}
\begin{cases}
\displaystyle
\mathbb D_t^f
p_{\boldsymbol\ell}^{\,f}(\mathbf a,t)
=
\frac{1}{2K}
\sum_{j=1}^d C_{\ell_j}
\Delta_{a_{\ell_j}}
p_{\boldsymbol\ell}^{\,f}(\mathbf a,t),
\\[3mm]
p_{\boldsymbol\ell}^{\,f}(\mathbf a,0)
=\delta_0(\mathbf a).
\end{cases}
\label{eq:generalized_GSRF_evolution}
\end{equation}
\end{proposition}

\begin{proof}
Conditionally on $\Upsilon^f(t)$, the vectors
$a_{\ell_j,\cdot}^{\,f}(t)$ are independent
centered Gaussian vectors with covariance
matrices
\[
\frac{C_{\ell_j}\Upsilon^f(t)}{K}I_{2\ell_j+1}.
\]
Consequently, their joint characteristic function is
\[
\widehat p_{\boldsymbol\ell}^{\,f}
(\boldsymbol\xi,t)
=
\mathbb E\exp\left\{
-\frac{\Upsilon^f(t)}{2K}
\sum_{j=1}^d
C_{\ell_j}\|\xi_{\ell_j}\|^2
\right\}.
\]
We recall that, the Laplace transform w.r.t. $t$ of
\[
q_f(t,\lambda)
:=\mathbb E[e^{-\lambda \Upsilon^f(t)}],
\]
is given by
\[
\int_0^\infty e^{-st}q_f(t,\lambda)\,dt
=
\frac{f(s)}{s(f(s)+\lambda)},\qquad s>0,
\]
and that the function $q_f(\cdot,\lambda)$ satisfies
\[
\mathbb D_t^f q_f(t,\lambda)
=-\lambda q_f(t,\lambda),
\qquad q_f(0,\lambda)=1.
\]
Applying this identity with
\(
\lambda=
\sum_{j=1}^d
C_{\ell_j}\|\xi_{\ell_j}\|^2/(2K),
\)
we obtain
\[
\mathbb D_t^f
\widehat p_{\boldsymbol\ell}^{\,f}
(\boldsymbol\xi,t)
=
-\frac{1}{2K}
\sum_{j=1}^d
C_{\ell_j}\|\xi_{\ell_j}\|^2
\widehat p_{\boldsymbol\ell}^{\,f}
(\boldsymbol\xi,t),
\]
which yields
\eqref{eq:generalized_GSRF_evolution}, by inverting the Fourier transform.

\end{proof}

\section{Acknowledgements} L.B. and D.M.  are associated to INdAM/GNAMPA (Istituto Nazionale di Alta Matematica “Francesco Severi”).
L.B. acknowledges financial support under NRRP, by the Italian Ministry of University and Research, funded by NextGenerationEU – 202277N5H9 - CUP: D53D23005670006 - Grant No. 973, 2023.
D.M. is supported by the MUR Departement of Excellence Programme 2023-2027 MatMod@Tov (CUP: E83C23000330006).
N.L.  would  like to thank the Australian Research Council for support through Discovery Grant DP220101680 and the Croatian Science Foundation (HRZZ) for support through the grant “Scaling in Stochastic Models”(IP-2022-10 8081).

\section{Appendices}\label{sec:appB}

\subsection{Appendix A: Relevant special cases of GSRFs with random angular power spectrum}\label{appA}
We describe here some relevant special cases of the GSRF, with random angular power spectrum, introduced in Subsection \ref{sectangul}, i.e. obtained by specifying the
multivariate subordinator with independent components. In these particular settings, it is possible to derive explicit
forms of the results given in theorems \ref{mixmom} and \ref{theoconv}. 

\begin{enumerate}
    \item \textbf{Multivariate stable subordinator with independent components and variable order}: in this case the Bernstein functions read $\varphi_i(\kappa)=\kappa^{\nu_i}$, for $\nu_i \in (0,1)$, for $i=1,..,d$, and thus \begin{equation}\Psi_\beta(\kappa_1,..,\kappa_d)=E_\beta(-\kappa_1^{\nu_1}-...-\kappa_d^{\nu_d}).\label{multi}
    \end{equation}

    Therefore, the GSRF has the following characteristic function of
any vector of spherical harmonic coefficients %
\begin{equation}
\Phi _{a_{\ell _{1}m_{1}},...,a_{\ell
_{d}m_{d}}}(\xi_{1},\xi_{2},...,\xi_{d})=E_\beta\left(-\sum_{i=1}^d\frac{C_{\ell_i}^{\nu_i}\xi_i^{2\nu_i}}{2^{\nu_i}}\right),
\end{equation}%
$\xi_i \in\mathbb{R}$, $i=1,...,d$, and, moreover,
\begin{equation*}\mathbb{E}e^{i\sum_{\ell=1}^L\langle\sqrt{\mathcal{H}_\ell(T_\beta)}\boldsymbol{Z}_\ell,\boldsymbol{\xi}_\ell\rangle}=E_\beta\left(-\sum_{\ell=1}^L\frac{C^{\nu_\ell}_{\ell}\Vert \boldsymbol{\xi}_\ell\Vert^{2\nu_\ell}}{2^{\nu_\ell}}\right),
\end{equation*}%
for $\beta \in (0,1)$, $\boldsymbol{\xi}_\ell \in\mathbb{R}^d$, $l=1,...,L$.

 In this case, the moments of any (integer) order of the r.v. $(\mathcal{H}_1(T_\beta),...,\mathcal{H}_d(T_\beta))$ are infinite, since the same holds for the multi-stable subordinator. As a consequence, $\mathbb{E}(a^{r_1}_{\ell_1 m_1}\cdots a^{r_d}_{\ell_d m_d})$ are infinite, for any $r_1,...,r_d>1$, while, for $m_i \neq m_j,$ \[\mathbb{E}(a_{\ell_1 m_1}\cdots a_{\ell_d m_d})=\mathbb{E}(\mathcal{H}^{1/2}_{\ell_1}(T_\beta)\cdots \mathcal{H}^{1/2}_{\ell_d}(T_\beta)) \mathbb{E}(Z_{\ell_1 m_1}\cdots Z_{\ell_d m_d})=0,\] if $\nu_i>1/2,$ for $i=1,...,d$. Indeed, under the last assumption,  \begin{eqnarray}
      \mathbb{E}(\mathcal{H}^{1/2}_{\ell_1}(T_\beta)\cdots \mathcal{H}^{1/2}_{\ell_d}(T_\beta))&=&\frac{\prod_{i=1}^d\Gamma(1-(2\nu_i)^{-1})}{\pi^{d/2}}\int_0^\infty t^{\frac{1}{2\nu_1}+...+\frac{1}{2\nu_d}}M_\beta(t)dt \notag\\
      &=&\frac{\prod_{i=1}^d\Gamma(1-(2\nu_i)^{-1})}{\beta \pi^{d/2}}\frac{\Gamma(\sum_{i=1}^d (2\nu_i)^{-1})}{\Gamma(\beta\sum_{i=1}^d (2\nu_i)^{-1})},\notag
 \end{eqnarray}
 by recalling the fractional moments of the stable subordinator (see \cite{WOL}) together with formula \eqref{mom}.

\item \textbf{Multivariate stable subordinator with independent components and variable scale}: in this case the Bernstein functions read $\varphi_i(\kappa)=c_i\kappa^{\nu}$, for $\nu \in (0,1)$, $c_i \in \mathbb R_+$, for $i=1,...,d$, and thus \begin{equation}\Psi_\beta(\kappa_1,..,\kappa_d)=E_\beta(-c_1\kappa_1^{\nu}-...-c_d\kappa_d^{\nu}).\label{multibis}
    \end{equation}
Note that, in this case, the convergence result given in Theorem \ref{theoconv} holds by choosing the sequence $\{r_\ell\}_{\ell >0}$ with $r_\ell=c_\ell^{1/\nu}$,
so that $\lim_{\ell \to \infty}\varphi_\ell(s/r_\ell)=\lim_{\ell \to \infty}c_\ell s^\nu/c_\ell =s^\nu$, for $s>0$. Thus, the convergence \eqref{convfi2} holds here for $V$ distributed as a $\nu$-stable r.v. with $\mathbb E e^{-sV}=e^{-s^\nu}.$

     \item \textbf{Multivariate inverse Gaussian subordinator with independent components}: this case is obtained by considering $\varphi_i(\kappa)=\sqrt{2\kappa+\gamma_i^2}-\gamma_i$, where $\gamma_i >0$, for $i=1,...,d.$ The mixed moments of the spherical harmonic coefficients can be evaluated by recalling Theorem \ref{mixmom} and considering that 
    $\varphi_i^{(j)}(0)=(-1)^{j-1 }(2j-3)!!\gamma_i^{1-2j}$ (where $(-1)!!=1$, by convention):
   \begin{eqnarray}&&\mathbb{E[}a_{\ell _{1}m_{1}}^{2n_1}\cdots a_{\ell _{d}m_{d}}^{2n_d}] \notag\\
   &=&\frac{(-1)^{\sum_{i=1}^dn_i}C^{n_1}_{\ell_1}\cdots C^{n_d}_{\ell_d}}{\gamma^{2n_1}_{\ell_1}\cdots \gamma^{2n_d}_{\ell_d}}\sum_{\boldsymbol{k}^{(1)}}\cdots \sum_{\boldsymbol{k}^{(i)}}A_{n_1,...,n_d}(\boldsymbol{k}^{(1)},...,\boldsymbol{k}^{(d)})\prod_{i=1}^d
\prod_{j_i=1}^{n_i}\left(\gamma_{\ell_i}(2j_i-3)!!\right)^{k^{(i)}_{j_i}},\notag\end{eqnarray}
so that, for example, 
     
\begin{equation}
\mathbb{E}(a^2_{\ell_1 m_1}a^2_{\ell_2 m_2})=\frac{2C_{\ell_1}C_{\ell_2}}{\gamma_{\ell_1}\gamma_{\ell_2}\Gamma(2\beta+1)} .\label{invgau}
\end{equation}

Here, the convergence of Theorem \ref{theoconv} holds, with $V\overset{d}{=}aT_\beta$, by choosing the sequence $\{r_\ell\}_{\ell >0}$ such that $\lim_{\ell \to \infty}r_\ell \gamma_\ell=1/a>0$ and $\lim_{\ell \to \infty}r^2_\ell \gamma^3_\ell=0$.

 
\item \textbf{Multivariate compound Poisson subordinator with independent components}: this case is obtained by considering $\varphi_i(\kappa)=\lambda_j(1-\Phi_i(\kappa))$, where, for any $i=1,...,d,$ $\lambda_i>0$ and $\Phi_i(\kappa):=\mathbb{E}e^{-\kappa U^{(i)}}$ is the Laplace transform of the distribution of the (i.i.d. and almost surely positive) addends in the definition of the $i$-th subordinator, i.e. $\mathcal{H}_i(t):=\sum_{k=1}^{N_i(t)}U_k^{(i)}$, under the assumption that the $i$-th Poisson process $N_i(t)$ (with parameter $\lambda_i$) is independent of the others and of the summands $U^{(i)}$, $i=1,...,d.$ 

By applying again Theorem \ref{mixmom} and considering that, in this case,  $\varphi_i^{(j)}(0)=-\lambda_i\Phi_i^{(j)}(0)=-\lambda_i\mathbb{E}(U^{(i)})^j$, we get, for example,

\[
\mathbb{E}(a^2_{\ell_1 m_1}a^2_{\ell_2 m_2})=\frac{2C_{\ell_1}C_{\ell_2}\lambda_{\ell_1}\lambda_{\ell_2}}{\Gamma(2\beta+1)}\mathbb{E}U^{(\ell_1)} \mathbb{E}U^{(\ell_2)}.\notag
\]
  \item \textbf{Multivariate Gamma subordinator with independent components}: this case is obtained by considering $\varphi_i(\kappa)=\gamma_i\log(1+\kappa/b_i)$, where $\gamma_i,b_i >0$, for $i=1,...,d.$ The moments can be evaluated, by recalling that for a Gamma univariate subordinator (with parameters $\gamma,b >0$) $\mathcal{H}$, we have that $\mathbb{E}\mathcal{H}(t)^{r}=(\gamma t)_r/b^r, $ where $(x)_n=\Gamma(x+n)/\Gamma(x)$, for $n \in \mathbb{N}, x \in \mathbb R\setminus\left\{0, -1,-2,... \right\}$, is the Pochhammer symbol:
    \begin{eqnarray}
&&\mathbb{E}\left(\mathcal{H}_1(T_\beta)^{r_1}\cdots\mathcal{H}_d(T_\beta)^{r_d}\right)=\frac{1}{\prod_{j=1}^d b_j^{r_j}}\int_0^\infty (\gamma_1t)_{r_1}\cdots(\gamma_dt)_{r_d}M_\beta(t)dt \notag \\
    &=&\frac{1}{\prod_{j=1}^d b_j^{r_j}}\sum_{k_1=0}^{r_1}\cdots\sum_{k_d=0}^{r_d}\gamma_1^{k_1}c(r_1,k_1)\cdots \gamma_d^{k_d}c(r_d,k_d)\int_0^\infty t^{k_1+...+k_d}M_\beta(t)dt \notag \\
    &=& \left[\text{by }\eqref{mom}\right] \label{stir} \\
    &=&\frac{1}{\prod_{j=1}^d b_j^{r_j}}\sum_{k_1=0}^{r_1}\cdots\sum_{k_d=0}^{r_d}\gamma_1^{k_1}c(r_1,k_1)\cdots \gamma_d^{k_d}c(r_d,k_d)\frac{(k_1+...+k_d)!}{\Gamma(\beta(k_1+...+k_d)+1)}\notag
    \end{eqnarray}
    where $c(n,k)$, $n,k \in \mathbb{N}, $ $k\leq  n$, are the unsigned Stirling numbers of the first kind. In the second step, we have applied the expansion of the Pochhammer symbol, i.e. $(x)_n=\sum_{k=0}^n c(n,k) x^k$
    (see \cite{ADA}).

    Therefore, we have that, for $d=2$, $r_1=r_2=1$, and $m_1 \neq m_2,$ 
    \begin{eqnarray}
\mathbb{E}(a^2_{\ell_1 m_1}a^2_{\ell_2 m_2})&=&\frac{C_{\ell_1}C_{\ell_2}}{b_{\ell_1}b_{\ell_2}}\left[ c^2(1,0)-\frac{c(1,0)c(1,1)(\gamma_{\ell_1}+\gamma_{\ell_2})}{\Gamma(\beta+1)}+\frac{2\gamma_{\ell_1}\gamma_{\ell_2}c^2(1,1)}{\Gamma(2\beta+1)} \right] \notag\\
&=& \frac{2\gamma_{\ell_1}\gamma_{\ell_2}C_{\ell_1}C_{\ell_2}}{b_{\ell_1}b_{\ell_2}\Gamma(2\beta+1)},    \label{stir2}
    \end{eqnarray}
since $c(1,0)=0$ and $c(1,1)=1.$ Formula \eqref{stir2} can be obtained, alternatively, from Theorem \ref{mixmom}, by considering that $\varphi_i^{(j)}(0)=(-1)^{j-1}\gamma_i/b_i^j,$ for any $i=1,...,d$, $j \in \mathbb N,$ and that $A_{1,...,1}(\boldsymbol{k}^{(1)},...,\boldsymbol{k}^{(d)})=d!(-1)^d/\Gamma(d\beta+1)$.

Note that, in this case, the convergence result given in Theorem \ref{theoconv} holds by choosing the sequence $\{r_\ell\}_{\ell >0}$ such that
\[\frac{\gamma_\ell}{r_\ell b_\ell}\to a>0 \qquad \text{and} \qquad \frac{\gamma_\ell}{r^2_\ell b^2_\ell}\to 0, \quad \ell \to \infty,\]
so that $\lim_{\ell \to \infty}\varphi_\ell(s/r_\ell)=\lim_{\ell \to \infty}\gamma_\ell \log(1+s/r_\ell b_\ell)=as$, for $s,a>0$. Thus, the convergence in \eqref{convfi2} holds with $V\overset{d}{=}aT_\beta$.
\end{enumerate}

\subsection{Appendix B: Relevant special cases of G-GSRF's with random global variance }

We review the choices of the function $\overline G$ in the definition of the G-GSRFs introduced in Section \ref{sec:GGSRF} which appear more relevant, discussing some of their most important properties.

\vspace{0.5cm}
\textbf{I) Examples of infinitely divisible G-GSRF}

Important examples arise when the mixing distribution belongs to classical
infinitely divisible families. In these cases, the resulting G-GSRF can
be interpreted as spherical analogues of well-known Gaussian variance mixtures.

These examples show that the class of G-GSRF's naturally unifies
Gaussian, variance-gamma, generalized hyperbolic and stable-type spherical
random fields within a common framework of Gaussian scale mixtures driven by
completely monotone functions.

\begin{enumerate}
    \item \textbf{Stable case}

If \(U\) is a positive stable random variable with index
\(\alpha\in(0,1)\), then
$
\overline G(s)=e^{-cs^\alpha},
$
and therefore
\[
\varphi_{\ell m}(\xi)
=
\exp\left\{
-c
\left(
\frac{C_\ell\xi^2}{2\mathbb EU}
\right)^\alpha
\right\}.
\]

In this case the spherical harmonic coefficients belong to the class of
elliptically symmetric stable distributions and possess power-law tails.

\item \textbf{Gamma case}

If the mixing variable \(U\) has Gamma distribution with parameters
\((\alpha,\lambda)\), then
\[
\overline G(s)
=
\left(1+\frac{s}{\lambda}\right)^{-\alpha}.
\]

Hence the characteristic function of the harmonic coefficients becomes
\[
\varphi_{\ell m}(\xi)
=
\left(
1+\frac{C_\ell\xi^2}
{2\lambda \mathbb EU}
\right)^{-\alpha}.
\]

The corresponding coefficients belong to the class of Gaussian variance
mixtures of variance-gamma type and possess polynomial tails.



\item
\textbf{Generalized inverse Gaussian case}

Let us recall the definition of the modified Bessel function of the second kind, i.e. \[
K_\lambda(x)
:=
\frac{1}{2}
\int_0^\infty
v^{\lambda-1}
\exp\left\{
-\frac{x}{2}
\left(v+\frac{1}{v}\right)
\right\}\,dv,
\qquad x>0,
\]
and let us consider the completely monotone function
\[
G(s)
=
\mathbb{E}e^{-sU}
=
\left(
\frac{\gamma^2}{\gamma^2+2s}
\right)^{\lambda/2}
\frac{
K_\lambda\left(
\delta\sqrt{\gamma^2+2s}
\right)
}{
K_\lambda(\delta\gamma)
},
\qquad s\geq0.
\]

Then $U$ follows a generalized inverse Gaussian distribution with
parameters $\lambda\in\mathbb{R}$ and $\delta,\gamma>0$, and with density
\[
f_U(u)
=
\frac{(\gamma/\delta)^\lambda}
{2K_\lambda(\delta\gamma)}
u^{\lambda-1}
\exp\left\{
-\frac{1}{2}
\left(\gamma^2u+\frac{\delta^2}{u}\right)
\right\},
\qquad u>0,
\]

Moreover,
\[
\mathbb{E}U
=
\frac{\delta}{\gamma}
\frac{K_{\lambda+1}(\delta\gamma)}
{K_\lambda(\delta\gamma)}.
\]

Therefore, according to \eqref{fi}, the characteristic function of each
spherical harmonic coefficient is
\[
\varphi_{\ell m}(\xi)
=
\left(
\frac{\gamma^2}
{\gamma^2+C_\ell\xi^2/
\mu_{\lambda,\delta,\gamma}}
\right)^{\lambda/2}
\frac{
K_\lambda\left(
\delta
\sqrt{
\gamma^2+C_\ell\xi^2/
\mu_{\lambda,\delta,\gamma}
}
\right)
}{
K_\lambda(\delta\gamma)
},
\]
where \(
\mu_{\lambda,\delta,\gamma}
:=
\delta K_{\lambda+1}(\delta\gamma)/\gamma
K_\lambda(\delta\gamma).
\)
Thus, the spherical harmonic coefficients are Gaussian variance
mixtures with generalized inverse Gaussian mixing distribution and
belong to the generalized hyperbolic family.

The inverse Gaussian case is recovered by taking
$\lambda=1/2$. Indeed,
\(
K_{1/2}(x)
=
\sqrt{\pi}e^{-x}/\sqrt{2x}
\)
and
\(
\mathbb{E}U=\delta/\gamma.
\)
Consequently, the characteristic function of the spherical harmonic
coefficients becomes
\[
\varphi_{\ell m}(\xi)
=
\exp\left\{
\delta\gamma
-
\delta
\sqrt{
\gamma^2+
\frac{\gamma C_\ell}{\delta}\xi^2
}
\right\}.
\]

\item 
\textbf{Logistic (hyperbolic-sine) case}

Let $V$ be a positive random variable with density
\[
f_V(v)
=
2\sum_{n=1}^{\infty}
(-1)^{n-1}n^2e^{-n^2v}
\mathds{1}_{(0,\infty)}(v).
\]
Equivalently, $V=U^2/2$, where $U$ has density
\[
f_U(u)
=
2u\sum_{n=1}^{\infty}
(-1)^{n-1}n^2e^{-n^2u^2/2}
\mathds{1}_{(0,\infty)}(u).
\]
The Laplace transform of $V$ is
\[
G(s)
=
\mathbb{E}e^{-sV}
=
\frac{\pi\sqrt{s}}{\sinh(\pi\sqrt{s})},
\qquad s\geq0.
\]
Indeed, using the infinite product representation of the hyperbolic
sine,
\[
\frac{\sinh(\pi\sqrt{s})}{\pi\sqrt{s}}
=
\prod_{n=1}^{\infty}
\left(1+\frac{s}{n^2}\right),
\]
we obtain
\[
G(s)
=
\prod_{n=1}^{\infty}
\left(1+\frac{s}{n^2}\right)^{-1}.
\]
In particular, $V$ is infinitely divisible. Moreover,
\(
\mathbb{E}V
=
-G'(0)
=\pi^2/6.
\)
Therefore, according to \eqref{fi}, the characteristic function of each
spherical harmonic coefficient is
\[
\varphi_{\ell m}(\xi)
=
G\left(
\frac{C_\ell\xi^2}{2\mathbb{E}V}
\right)=
\frac{\sqrt{3C_\ell}\,|\xi|}
{\sinh\left(\sqrt{3C_\ell}\,|\xi|\right)}.
\]
Thus, in this case the characteristic function of the spherical
harmonic coefficients admits a simple explicit representation in
terms of the hyperbolic sine function.

\end{enumerate}

\textbf{II) Examples of non-infinitely divisible G-GSRFs}
\begin{enumerate} \item \textbf{Two-parameters Mittag-Leffler case}

If we consider the function $E_{\beta,\alpha}(z):=\sum_{n=0}^\infty \frac{z^n}{\Gamma(\beta n+\alpha)}$, for $z \in \mathbb C,$ $\alpha, \beta >0,$ it is well-known that $E_{\beta,\alpha}(-x)$, $x>0,$ is completely monotone on $[0,\infty)$ if and only if $0<\beta\leq 1$, $\alpha \geq \beta$  (see \cite{SCH2}). Moreover, in the Laplace-Stieltjes representation, 
\begin{equation}\overline G(x)=\Gamma(\alpha)E_{\beta,\alpha}(-x)=\int_0^\infty e^{-xu}d\overline \mu_{\beta, \alpha}(u), \qquad x \geq 0, \label{g}\end{equation}
and the underlying probability measure $\overline \mu_{\beta,\alpha}$ on $[0,\infty)$ has the following density function
\begin{equation}f_{\beta,\alpha}(x)=\Gamma(\alpha)\sum_{k=0}^\infty \frac{(-x)^k}{k!\Gamma(-\beta k+\alpha-\beta)}=\Gamma(\alpha)W_{-\beta, \alpha-\beta}(-x), \qquad x \geq 0,\label{g2}\end{equation}
where $W_{a,b}(\cdot)$ denotes the Wright function with parameters $a$, $b$.

Moreover, the characteristic function of each spherical harmonic
coefficient can be explicitly obtained. Indeed, since
\(
\mathbb{E}U
=
\Gamma(\alpha)/\Gamma(\alpha+\beta),
\)
it follows from \eqref{deft} that
\[
\varphi_{\ell m}(\xi)
=
G\left(
\frac{C_\ell\xi^2}{2\mathbb{E}U}
\right)=
\Gamma(\alpha)
E_{\beta,\alpha}\left(
-\frac{\Gamma(\alpha+\beta)}
{2\Gamma(\alpha)}
C_\ell\xi^2
\right),
\qquad \xi\in\mathbb{R}.
\]
In particular, for $\alpha=1$ and $0<\beta<1$, this reduces to
\[
\varphi_{\ell m}(\xi)
=
E_\beta\left(
-\frac{\Gamma(1+\beta)}{2}
C_\ell\xi^2
\right).
\]
while for $\alpha=\beta=1$, we recover the Gaussian characteristic function.

Thus, according to \eqref{fi}, we have that, in this case
\begin{eqnarray}
    \mathbb Ee^{i(\xi_1\mathcal{T}(x_1)+...+\xi_d \mathcal{T}(x_d))}=\Gamma(\alpha)E_{\beta, \alpha}\left(-\frac{\Gamma(\alpha+\beta)}{2 \Gamma(\alpha)}\sum_{i,j=1}^d\xi_i\xi_j\rho_{i,j}\right), \notag
\end{eqnarray}
since $\mathbb EU=-\Gamma(\alpha) E'_{\beta,\alpha}(0)=\Gamma(\alpha)/\Gamma(\alpha+\beta).$
\begin{remark}
We have the following limiting cases:
\begin{itemize}
    \item for $0<\beta<1$ and $\alpha=1$: we recover the one-parameter Mittag-Leffler case, since \eqref{g} coincides with \eqref{ML} and \eqref{g2} reduces to \eqref{mai}. 

    \item for $\beta=\alpha=1$: the measure $\mu_{\beta, \alpha}$ coincides with the Dirac measure with pole in $1$.

    \item for $\beta=1$ and $\alpha>1$: the measure $\overline\mu_{\beta, \alpha}$ reduces to a beta-type distribution, with density
    \[f_{1,\alpha}(x)=(\alpha-1)(1-x)^{\alpha-1}\mathds{1}_{(0,1)}(x).\]

\end{itemize}
\end{remark}

\item \textbf{Modified Bessel case}

Let us consider the modified Bessel function of the second kind of order
$\lambda>0$, defined as
\[
K_\lambda(x)
=
\frac{1}{2}\int_0^\infty
u^{\lambda-1}
\exp\left\{-\frac{x}{2}\left(u+\frac{1}{u}\right)\right\}\,du,
\qquad x>0.
\]
Let moreover the definition \eqref{deft} hold with $U$ being an inverse Gamma
random variable with density
\begin{equation}
f_U(u)
=
\frac{\beta^\lambda}{\Gamma(\lambda)}
u^{-\lambda-1}e^{-\beta/u}
\mathbf{1}_{(0,\infty)}(u),
\qquad \lambda,\beta>0.
\tag{9.9}
\end{equation}
Its Laplace transform can be written in terms of the modified Bessel
function as
\[
G(s)
=
\mathbb{E}e^{-sU}
=
\frac{2}{\Gamma(\lambda)}
(\beta s)^{\lambda/2}
K_\lambda\left(2\sqrt{\beta s}\right),
\qquad s\geq0.
\]
Notice that
\(
\mathbb{E}U
=
\frac{\beta}{\lambda-1},
\lambda>1,
\)
whereas $\mathbb{E}U=+\infty$ for $0<\lambda\leq1$. Hence, in the
finite-variance framework considered in \eqref{deft}, we restrict ourselves
to $\lambda>1$.

The characteristic function of each spherical harmonic coefficient
can then be explicitly evaluated. Indeed, 
\[
\varphi_{\ell m}(\xi)
=
G\left(
\frac{C_\ell\xi^2}{2\mathbb{E}U}
\right)=
\frac{2}{\Gamma(\lambda)}
\left(
\frac{(\lambda-1)C_\ell\xi^2}{2}
\right)^{\lambda/2}
K_\lambda\left(
\sqrt{2(\lambda-1)C_\ell}\,|\xi|
\right),
\qquad \xi\in\mathbb{R}.
\]
We remark that, after the normalization by $\mathbb{E}U$, the
characteristic function of the spherical harmonic coefficients does
not depend explicitly on the scale parameter $\beta$.

Moreover, according to \eqref{fi}, the joint characteristic function of
the random vector $(\mathcal{ T}(x_1),\ldots,\mathcal T(x_d))$ is
\[
\begin{aligned}
& \mathbb Ee^{i(\xi_1\mathcal{T}(x_1)+...+\xi_d \mathcal{T}(x_d))} =
G\left(
\frac{1}{2\mathbb{E}U}
\sum_{i,j=1}^d\xi_i\xi_j\rho_{i,j}
\right)\\
&\quad =
\frac{2}{\Gamma(\lambda)}
\left(
\frac{\lambda-1}{2}
\sum_{i,j=1}^d\xi_i\xi_j\rho_{i,j}
\right)^{\lambda/2}
K_\lambda\left(
\sqrt{
2(\lambda-1)
\sum_{i,j=1}^d\xi_i\xi_j\rho_{i,j}
}
\right).
\end{aligned}
\]

\end{enumerate}

\end{document}